\documentclass{article}
    
\usepackage{arxiv}

\usepackage[utf8]{inputenc} 
\usepackage[T1]{fontenc}    
\usepackage{hyperref}       
\usepackage{url}            
\usepackage{booktabs}       
\usepackage{amsfonts}       
\usepackage{nicefrac}       
\usepackage{microtype}      
\usepackage{lipsum}

\usepackage[utf8]{inputenc}
\usepackage{mathcomp}
\usepackage{mathtools}
\usepackage{amssymb}
\usepackage{tikz}

\usepackage{amsthm}

\usetikzlibrary{
  knots,
  hobby,
  celtic,
  decorations.pathreplacing,
  shapes.geometric,
  calc,
  braids
}

\usetikzlibrary{shapes}

\usepackage{float}

\usepackage{amsmath,amssymb,amsfonts,epsfig}

\usepackage{adjustbox}

\tikzset{
  knot diagram/every strand/.append style={
    ultra thick,black
  },
  show curve controls/.style={
    postaction=decorate,
    decoration={show path construction,
      curveto code={
        \draw [blue, dashed]
        (\tikzinputsegmentfirst) -- (\tikzinputsegmentsupporta)
        node [at end, draw, solid, red, inner sep=2pt]{};
        \draw [blue, dashed]
        (\tikzinputsegmentsupportb) -- (\tikzinputsegmentlast)
        node [at start, draw, solid, red, inner sep=2pt]{}
        node [at end, fill, blue, ellipse, inner sep=2pt]{}
        ;
      }
    }
  },
  show curve endpoints/.style={
    postaction=decorate,
    decoration={show path construction,
      curveto code={
        \node [fill, blue, ellipse, inner sep=2pt] at (\tikzinputsegmentlast) {}
        ;
      }
    }
  }
}

\newcommand{\R}{\mathbb{R}} 
\newcommand{\Z}{\mathbb{Z}} 

\newtheorem{thm}{Theorem}[section]
\newtheorem{thrm}[thm]{Theorem}
\newtheorem{defn}[thm]{Definition}
\newtheorem{exa}[thm]{Example}

\newtheorem{lemma}[thm]{Lemma}
\newtheorem{prop}[thm]{Proposition}
\newtheorem{cor}[thm]{Corollary}
\newtheorem{rem}[thm]{Remark}
\newtheorem{conj}[thm]{Conjecture}

\renewcommand{\le}{\leqslant}  
\renewcommand{\ge}{\geqslant}

\newcommand{\e}{\varepsilon}

\newcommand{\de}{\delta}

\newcommand{\KP}[1]{%
  \begin{tikzpicture}[baseline=-\dimexpr\fontdimen22\textfont2\relax]
  #1
  \end{tikzpicture}%
}
\newcommand{\KPA}{
  \KP{\filldraw[color=black, fill=none,ultra thick] circle (0.3);}%
}
\newcommand{\KPB}{
  \KP{
    \draw[color=black,ultra thick] (-0.3,0.3) -- (0.3,-0.3);
    \draw[color=black,ultra thick] (-0.3,-0.3) -- (-0.05,-0.05);
    \draw[color=black,ultra thick] (0.05,0.05) -- (0.3,0.3);
  }%
}

\newcommand{\KPC}{
  \KP{%
    \draw[color=black,ultra thick] (-0.3,0.3) .. controls (0,0) .. (0.3,0.3);
    \draw[color=black,ultra thick] (-0.3,-0.3) .. controls (0,0) .. (0.3,-0.3);
  }%
}

\newcommand{\KPD}{
  \KP{%
    \draw[color=black,ultra thick] (-0.3,-0.3) .. controls (0,0) .. (-0.3,0.3);
    \draw[color=black,ultra thick] (0.3,-0.3) .. controls (0,0) .. (0.3,0.3);
  }%
}

\newcommand{\KPF}{
  \KP{%
    \draw[color=black, ultra thick] (-0.3,-0.3) .. controls (0.05,0) .. (-0.3,0.3);
  }%
}

\newcommand{\PC}{
  \KP{
     \draw[color=black, ultra thick, ->] (-0.3,-0.3) -- (0.3,0.3);
     \draw[color=black,ultra thick, ->]  (-0.08,0.08) -- (-0.3,0.3);
    \draw[color=black,ultra thick] (0.08,-0.08) -- (0.3,-0.3);
  }%
}

\newcommand{\NC}{
  \KP{
     \draw[color=black, ultra thick, ->] (0.3,-0.3) -- (-0.3,0.3);
     \draw[color=black,ultra thick, ->]  (0.08,0.08) -- (0.3,0.3);
    \draw[color=black,ultra thick] (-0.08,-0.08) -- (-0.3,-0.3);
  }%
}

\newcommand{\SC}{
  \KP{
     \draw[color=black, ultra thick, ->] (0.3,-0.3) -- (-0.3,0.3);
     \draw[color=black,ultra thick, ->]  (-0.3,-0.3) -- (0.3,0.3);
     \filldraw [black] (0,0) circle (2pt);
  }%
}

\newcommand{\TU}{
  \KP{
     \draw[color=black, ultra thick, ->] (-0.3,-0.3) .. controls (0,-0.2) and (0,0.2) .. (-0.3,0.3);
     \draw[color=black, ultra thick, ->] (0.3,-0.3) .. controls (0,-0.2) and (0,0.2) .. (0.3,0.3);
  }%
}

\newcommand{\unkt}{
  \KP{
     \draw[color=black, ultra thick, ->] (0,0.3) arc (90:450:0.3);
  }%
}

\newcommand{\unkd}{
  \KP{
     \draw[color=black, ultra thick, ->] (0,-0.3) arc (-90:90:0.3);
  }%
}

\newcommand{\finfty}{
  \KP{
     \draw[color=black, ultra thick,] (0,0.3) arc (90:450:0.3);
      \draw[color=black, ultra thick,] (0.6,0.3) arc (90:450:0.3);
      \filldraw [black] (0.3,0) circle (2pt) node[anchor=west] {}; 
  }%
}

\newcommand{\RRKF}{
  \KP{
      \draw[color=red, ultra thick,] (-6,-6) -- (-6,6) ;
    \draw[color=red, ultra thick,] (-6,0) -- (-6,0) node[anchor=east] {H};
      
       \draw[color=red, ultra thick,] (0,-6) -- (0,0.2);
       \draw[color=red, ultra thick,] (0,0.6) -- (0,1.5);
        \draw[color=red, ultra thick,] (0,1.9) -- (0,3.1);
       \draw[color=red, ultra thick,] (0,3.5) -- (0,4.4);
       \draw[color=red, ultra thick,] (0,4.8) -- (0,6);
       \draw[color=red, ultra thick,] (0,0) -- (0,0) node[anchor=east] {L};

        \draw[color=gray, thick,] (6,-6) -- (6,6);
         \draw[color=gray,  thick, dashed] (0,6) ellipse (6 and 1);
          \draw[color=gray,  thick, dashed] (0,-6) ellipse (6 and 1);
          
    \draw[color=black, ultra thick] (0,6) .. controls (-1, 5.8) and (-0.9, 5.6)  .. (-0.2,5.4);
    \draw[color=black, ultra thick] (0.1,5.3) .. controls (2, 4.5) and (-4, 4.2)  .. (-0.2,4);
    \draw[color=black, ultra thick] (0.1,3.95) .. controls (2, 3.5)  .. (-0.2,3.2);
    
     \draw[color=black, ultra thick] (-0.2,3.2) .. controls (-3, 3) and (-2.9, 2.8)  .. (-0.2,2.6);
    \draw[color=black, ultra thick] (0.1,2.5) .. controls (6, 2) and (-5, 1.5)  .. (-0.2,1.2); 
    \draw[color=black, ultra thick] (0.1,1.15) .. controls (2, 0.7)  .. (-0.4,0.3);
     \draw[color=black, ultra thick] (-0.4,0.3) .. controls (-5,-0.1) and (-4, 4.7) .. (-0.2,3.2);

    \draw[color=black, ultra thick] (-0.2,3.2) -- (-0.05,3.05);
    \draw[color=black, ultra thick] (0.1,3) .. controls (0.8,3)  .. (1,2.5);
     \draw[color=black, ultra thick] (1.1,2.3) --(1.2,2);
     \draw[color=black, ultra thick] (1.25,1.8) .. controls (1.3,1.7)  .. (1,1.1);
      \draw[color=black, ultra thick] (0.9,0.8) --(0.8,0.6);
         \draw[color=black, ultra thick] (0.7,0.3) --(0.1,-0.6);
         \draw[color=black, ultra thick] (-0.2,-0.9) .. controls (-4, -5) and (-5, -5.5)  .. (-6,-6);

    \filldraw [black] (-0.2,3.2) circle (2pt)  node[anchor=south east] {s};
     \filldraw [red] (0,6) circle (2pt) node[anchor= south west] {l};
     \filldraw [red] (-6,-6) circle (2pt) node[anchor=south east] {h}; 
  }%
}

\newcommand{\RRKFL}{
  \KP{
      \draw[color=red, ultra thick,] (-6,-1) -- (-6,6) ;
    \draw[color=red, ultra thick,] (-6,2) -- (-6,2) node[anchor=east] {H};
      
       \draw[color=red, ultra thick,] (0,-1) -- (0,0.2);
       \draw[color=red, ultra thick,] (0,0.6) -- (0,1.5);
        \draw[color=red, ultra thick,] (0,1.9) -- (0,3.1);
       \draw[color=red, ultra thick,] (0,3.5) -- (0,4.4);
       \draw[color=red, ultra thick,] (0,4.8) -- (0,6);
       \draw[color=red, ultra thick,] (0,2) -- (0,2) node[anchor=east] {L};

        \draw[color=gray, thick,] (6,-1) -- (6,6);
         \draw[color=gray,  thick, dashed] (0,6) ellipse (6 and 1);
          \draw[color=gray,  thick, dashed] (0,-1) ellipse (6 and 1);
          
    \draw[color=black, ultra thick] (0,6) .. controls (-1, 5.8) and (-0.9, 5.6)  .. (-0.2,5.4);
    \draw[color=black, ultra thick] (0.1,5.3) .. controls (2, 4.5) and (-4, 4.2)  .. (-0.2,4);
    \draw[color=black, ultra thick] (0.1,3.95) .. controls (2, 3.5)  .. (-0.2,3.2);
    
     \draw[color=black, ultra thick] (-0.2,3.2) .. controls (-3, 3) and (-2.9, 2.8)  .. (-0.2,2.6);
    \draw[color=black, ultra thick] (0.1,2.5) .. controls (6, 2) and (-5, 1.5)  .. (-0.2,1.2); 
    \draw[color=black, ultra thick] (0.1,1.15) .. controls (2, 0.7)  .. (-0.4,0.3);
     \draw[color=black, ultra thick] (-0.4,0.3) .. controls (-4, -0.4)  .. (-6,-1);

     \filldraw [red] (0,6) circle (2pt) node[anchor= south west] {l};
     \filldraw [red] (-6,-1) circle (2pt) node[anchor=south east] {h};
      \filldraw [black] (2.5,2) circle (0.00000001pt) node[anchor=south east] {c};
  }%
}

\newcommand{\trsf}{
  \KP{
      \draw (0,0) circle (3cm);
      \draw[dashed] (0,0) ellipse (3cm and 1cm);
      
      \draw[color=blue, thick] (-0.2,3.3) .. controls (-0.1,3.5) and (0.1,3.1) .. (0.2,3.3);
      \draw[color=blue, thick] (-0.2,3.3) .. controls (-0.1,3.1) and (0.1,3.5) .. (0.2,3.3);

      \draw[color=black, ultra thick,] (-2.5,0) -- (-1.5,0);
     \draw[color=black, ultra thick,] (-1,0) -- (2,0); 
     \draw[color=black, ultra thick] (2,0) .. controls (3,0) and (3.25,0.07)  .. (2.5,0.5);     
     \draw[color=black, ultra thick] (2.5,0.5) .. controls (0.5,2) and (-1,2) .. (-1.25,0);
     \draw[color=black, ultra thick] (-1.25,0) .. controls (-1.25,-1) and (-0.5,-2)  .. (0.65,-0.2);
     \draw[color=black, ultra thick,] (0.85,0.2) -- (1.2,0.7);
     \filldraw [red] (1.2,0.7) circle (2pt)  node[anchor=east] {h};
     \filldraw [red] (-2.5,0) circle (2pt) node[anchor=east] {l};
      
      \draw[color=black, ultra thick, ->] (4,0) --(6,0);
      
   \draw (10,0) circle (3cm);
      \draw[dashed] (10,0) ellipse (3cm and 1cm);   
      
      \draw[color=blue, thick] (9.8,3.3) .. controls (9.9,3.5) and (10.1,3.1) .. (10.2,3.3);
      \draw[color=blue, thick] (9.8,3.3) .. controls (9.9,3.1) and (10.1,3.5) .. (10.2,3.3);
      
      \draw[color=black, ultra thick, ->] (9,-3.5) --(7.7,-4.8);
      
      \draw[color=black, ultra thick,] (12.5,0) -- (11.5,0);
     \draw[color=black, ultra thick,] (11,0) -- (8,0); 
     \draw[color=black, ultra thick] (8,0) .. controls (7,0) and (6.75,-0.07)  .. (7.5,-0.5);     
     \draw[color=black, ultra thick] (7.5,-0.5) .. controls (9.5,-2) and (11,-2) .. (11.25,0);
     \draw[color=black, ultra thick] (11.25,0) .. controls (11.25,1) and (10.5,2)  .. (9.35,0.2);
     \draw[color=black, ultra thick,] (9.15,-0.2) -- (8.8,-0.7);
     \filldraw [red] (12.5,0) circle (2pt)  node[anchor=north] {h};
     \filldraw [red] (8.8,-0.7) circle (2pt) node[anchor=east] {l};

       \draw[dashed] (5,-7)circle (3cm);
       
        \draw[color=black, ultra thick,] (8,-7) -- (6.5,-7);
     \draw[color=black, ultra thick,] (6,-7) -- (3,-7); 
     \draw[color=black, ultra thick] (3,-7) .. controls (2,-7) and (1.75,-7.07)  .. (2.5,-7.5);     
     \draw[color=black, ultra thick] (2.5,-7.5) .. controls (4.5,-9) and (6,-9) .. (6.25,-7);
     \draw[color=black, ultra thick] (6.25,-7) .. controls (6.25,-6) and (5.5,-5)  .. (4.35,-6.8);
     \draw[color=black, ultra thick,] (4.15,-7.2) -- (3.8,-7.7);
     \filldraw [red] (8,-7) circle (2pt)  node[anchor=west] {h};
     \filldraw [red] (3.8,-7.7) circle (2pt) node[anchor=east] {l};

     \filldraw [blue] (0,3) circle (2pt) node[anchor=east] {};
     \filldraw [blue] (10,3) circle (2pt) node[anchor=east] {};
  }%
}

\newcommand{\RTHREE}{
  \KP{
      \draw[color=black, ultra thick,] (1,-1) -- (-1,1);
      \draw[color=black, ultra thick,] (-1,-1) -- (-0.2,-0.2);
      \draw[color=black, ultra thick,] (0.2,0.2) -- (1,1);
      \draw[color=black, ultra thick,] (0.3,0.5) .. controls (0.1,0.8) and (-0.1,0.8) .. (-0.3,0.5);
      \draw[color=black, ultra thick,] (0.5,0.3) -- (0.8,0);
      \draw[color=black, ultra thick,] (-0.5,0.3) -- (-0.8,0);

     \draw[color=black, ultra thick,<->] (2,0) -- (3,0);
     
     \draw[color=black, ultra thick,] (6,-1) -- (4,1);
      \draw[color=black, ultra thick,] (4,-1) -- (4.8,-0.2);
      \draw[color=black, ultra thick,] (5.2,0.2) -- (6,1);
      
      \draw[color=black, ultra thick,] (5.3,-0.5) .. controls (5.1,-0.8) and (4.9,-0.8) .. (4.7,-0.5);
      \draw[color=black, ultra thick,] (5.5,-0.3) -- (5.8,0);
      \draw[color=black, ultra thick,] (4.5,-0.3) -- (4.2,0);
     
  }%
}

\newcommand{\SRTHREE}{
  \KP{
      \draw[color=black, ultra thick,] (0.5,-0.5) -- (-1.5,1.5);
     \draw[color=black, ultra thick,] (-0.5,-0.5) -- (1.5,1.5);
     \draw[color=black, ultra thick,] (-0.8,1) -- (0.8,1);
     \draw[color=black, ultra thick,] (1.2,1) -- (2,1);
     \draw[color=black, ultra thick,] (-1.2,1) -- (-2,1);
    
     \draw[color=black, ultra thick,<->] (3,0.5) -- (4,0.5);

     \filldraw [black] (0,0) circle (2pt) node[anchor=north west] {};
     
     \draw[color=black, ultra thick,] (8.5,-0.5) -- (6.5,1.5);
     \draw[color=black, ultra thick,] (5.5,-0.5) -- (7.5,1.5);
     \draw[color=black, ultra thick,] (6.2,0) -- (7.8,0);
     \draw[color=black, ultra thick,] (8.2,0) -- (9,0);
     \draw[color=black, ultra thick,] (5.8,0) -- (5,0);
     
     \filldraw [black] (7,1) circle (2pt) node[anchor=north west] {};

  }%
}

\newcommand{\SFLIP}{
  \KP{
    
    \draw[color=black, ultra thick,] (0,0) .. controls (-1,0.5) and (1.2,1) .. (1,1.5);
    \draw[color=black, ultra thick,] (0,0) .. controls (0.5,0.2)  .. (0.1,0.5);
    \draw[color=black, ultra thick,] (-1,1.5) -- (-0.2,0.75);
    
    \draw[color=black, ultra thick,] (-1,-1.5) -- (-0.2,-0.75);  
    \draw[color=black, ultra thick,] (0,0) .. controls (0.5,-0.2)  .. (0.1,-0.5);
     \draw[color=black, ultra thick,] (0,0) .. controls (-1,-0.5) and (1.2,-1) .. (1,-1.5);

    \filldraw [black] (0,0) circle (3pt) node[anchor=north west] {};
     \draw[color=black, ultra thick,<->] (3,0) -- (4,0);
     
        \draw[color=black, ultra thick,] (6,-1.5) -- (8,1.5);
       \draw[color=black, ultra thick,] (8,-1.5) -- (6,1.5);
       
           \filldraw [black] (7,0) circle (3pt) node[anchor=north west] {};

  }%
}

\newcommand{\RONE}{
  \KP{
      \draw[color=black, ultra thick,] (-0.3,-0.3) -- (-1,-1);
     \draw[color=black, ultra thick,] (-1,1) .. controls (1,-2) and (1,1) .. (0,0);
    
     \draw[color=black, ultra thick,<->] (2,0) -- (3,0);
     
    \draw[color=black, ultra thick,] (4,-1) .. controls (6,-1) and (6,1) .. (4,1);

  }%
}

\newcommand{\RTWO}{
  \KP{
      \draw[color=black, ultra thick,] (-1,-1) .. controls (1,-1) and (1,1) .. (-1,1);
      
      \draw[color=black, ultra thick,] (1,1) -- (0.5,0.7);
      \draw[color=black, ultra thick,] (1,-1) -- (0.5,-0.7);
      
      \draw[color=black, ultra thick,] (0.2,0.5) .. controls (0,0.4) and (0,-0.4) .. (0.2,-0.5);
     \draw[color=black, ultra thick,<->] (2,0) -- (3,0);
     \draw[color=black, ultra thick,] (4,-1) .. controls (5,-1) and (5,1) .. (4,1);
     
      \draw[color=black, ultra thick,] (6,-1) .. controls (5,-1) and (5,1) .. (6,1);

  }%
}

\title{Finite type invariants for knotoids}

\author{
  Manousos Manouras \\
  School of Applied Mathematical and Physical Sciences \\
  National Technical University of Athens \\
   GR-15780 Athens, Greece \\
  \texttt{manousosmanouras@hotmail.com} \\
   \And
 Sofia Lambropoulou \\
  School of Applied Mathematical and Physical Sciences \\
  National Technical University of Athens \\
   GR-15780 Athens, Greece \\
  \texttt{sofia@math.ntua.gr} \\
   \And
 Louis H. Kauffman \thanks{Kauffman’s work was supported by the Laboratory of Topology and Dynamics, Novosibirsk State University (contract no.14.Y26.31.0025 with the Ministry of Education and Science of the Russian Federation).
} \\
  Department of Mathematics, Statistics and Computer
Science \\
  University of Illinois at Chicago \\
   851 South Morgan St., Chicago
IL 60607-7045, U.S.A. \\
and \\ 
Department of Mechanics and Mathematics \\ 
Novosibirsk State University \\ 
Novosibirsk, Russia \\
  \texttt{kauffman@math.uic.edu} \\
}

\begin{document}
\maketitle

\begin{abstract}
We extend the theory of Vassiliev (or finite type) invariants for knots to knotoids using two different approaches. Firstly, we take  closures on knotoids to obtain knots and we use the Vassiliev invariants for knots, proving that these are knotoid isotopy invariant. Secondly, we define finite type invariants directly on knotoids, by extending  knotoid invariants to singular knotoid invariants via the Vassiliev skein relation. Then, for spherical knotoids we  show that there are non-trivial type-1 invariants, in contrast with classical knot theory where type-1 invariants vanish.  We give a complete theory of type-1 invariants for spherical knotoids, by classifying linear chord diagrams of order one, and we present examples arising from the affine index polynomial and the extended bracket polynomial.

\end{abstract}


\section*{Introduction}

This paper initiates the study of Vassiliev or finite type invariants for knotoids. In this introductory section  a brief summary is made of what will be discussed in detail in the main part of the paper. Precise definitions and rigorous proofs are to be found in the sections that follow. 

The theory of knotoids,  introduced by V.~Turaev \cite{turaev2012knotoids}, is an extension of classical knot theory. Roughly, a knotoid diagram  in a surface $\Sigma$ is an open-ended knot diagram, with the two endpoints lying in any regions of $\Sigma$, while a knotoid is an equivalence class of knotoid diagrams, under the Reidemeister moves, so that  the endpoints (the leg and the head) remain in their regions (for precise definitions the reader is referred to Subsection~\ref{sub:knotoids}). Joining the endpoints of a knotoid diagram by a simple arc which crosses transversally other arcs of the knotoid, gives rise to a classical knot diagram. The height of a knotoid is the minimal number of intersections of all joining arcs over all knotoid diagrams of the knotoid. Turaev showed that the set of classical knots injects faithfully into the set of spherical knotoids. Knotoids have caught the attention of the mathematical community in the last years with many different results and with applications to other aspects of science, such as the topological study of proteins, \cite{Bartholomew, GUGUMCU2017186, Goundaroulis_2017, Goundaroulisprotein, Dorier_2018, Goundaroulis2019ASC, Gugumcu2016ASO, Braidoids, RailKnotoids2019, Barbensi, BarbensiBuck, Gugumcu2018BiquandleCI, Adams_Henrich, doi:10.1142/4256}.

\begin{figure}[H]
    \centering
    \includegraphics[scale=1]{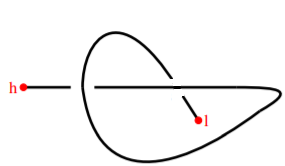}
    \caption{A knotoid diagram}
    \label{fig:ex_knotoid}
\end{figure}

The Vassiliev invariants \cite{Vassiliev1990} are extensions of classical knot invariants to singular knots, that is, knots which fail to be embeddings in the 3-space in finitely many transversal double points. The extension is possible using the relation called the Vassiliev skein relation: $v(\SC)=v(\PC)-v(\NC)$. (For precise definitions the reader is referred to Subsection~\ref{sub:fti_knots}). These invariants behave like building blocks for classical knot invariants, like the  Jones polynomial (see~\cite{jones1985}) whose coefficients in the Taylor expansion are Vassiliev invariants. It is indeed a very strong family of knot invariants, as shown by the proof of the Vassiliev-Kontsevich Theorem \cite{Kontsevich1993}. Nevertheless, it is not yet known whether Vassiliev invariants can classify knots or even if they detect the unknot.

 Trying to extend the theory of  finite type invariants for knots to knotoids one can use two different approaches. The immediate one is to  take various types of closures on singular knotoid diagrams (see Figure~\ref{fig:sing_knotoid} for an example; for precise definitions see Definitions~\ref{defn:singular_knotoid_diagram} and~\ref{defn:rvi}), namely  the under (or over) closure, the virtual closure and a special type of singular closure, to obtain singular knots or singular virtual knots. 
 Then we apply on the resulting  knots (or virtual knots)  known Vassiliev invariants for knots (or virtual knots).  These will be finite type invariants for (virtual) knotoids too.   
 In general,  both classical knot theory and virtual knot theory (introduced by L.H.~Kauffman~\cite{KAUFFMAN1999}) contribute to the the invariants of knotoids by applying classical invariants to the under (or over) closure of the knotoid, and by applying virtual invariants to the virtual closure of the knotoid. 
However, no type of closure of knotoids classifies knotoids; one can easily construct examples of non-equivalent knotoids having isotopic closures.

  The second approach is to define finite type invariants directly on knotoids, by extending  knotoid invariants to singular knotoid invariants via the Vassiliev skein relation. This approach is our main focus.

\begin{figure}[H]
    \centering
    \includegraphics[scale=1]{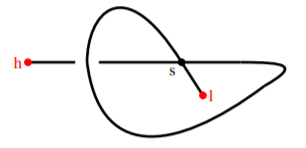}
    \caption{A singular knotoid diagram}
    \label{fig:sing_knotoid}
\end{figure}

 In this work we consider mainly knotoid diagrams in the plane and in the 2-sphere and we show that there are non-trivial Vassiliev invariants of type-1 for such knotoids. This is in contrast with classical knot theory where type-1 invariants vanish. See Remark~\ref{rem:v_1_knot_trivial}. The first tool in our theory is the introduction of linear chord diagrams, which extend the combinatorial approach of circular chord diagrams for singular knots. A linear chord diagram of order  $n$  is an oriented closed interval, equipped with a distinguished set of $n$ disjoint pairs of distinct points being connected with immersed unknotted arcs in the sphere or in the plane in general position, the `chords' (see also Definition~ \ref{defn:lcd}). See Fig.~\ref{fig:lcd} for examples. 

\begin{figure}[H]

 \centering
\includegraphics[height=3.5cm]{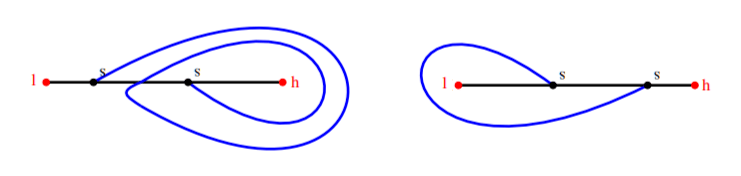}
 
  \caption{Linear chord diagrams}
  
  \label{fig:lcd}
\end{figure}

 We give a complete theory of Vassiliev invariants of type-1 for spherical knotoids, by classifying linear chord diagrams of order one.  As we show, a spherical linear chord diagram of order one can be abstracted by its winding number, as illustrated in Fig.~\ref{fig:alcd}. We then classify spherical knotoid diagrams with one singular crossing, up to singular equivalence. Singular equivalence comprises rigid vertex isotopy and classical crossing switches. Switching classical crossings is equivalent to considering flat diagrams (homotopy classes of diagrams with no extra information over/under at crossings) or descending diagrams (diagrams which `descend' as traversed starting from specific points; see Definition~\ref{defn:descending}). In this paper we take the approach of descending diagrams, as it allows us to make use of the lifting of knotoid diagrams to space curves. The reader is referred to Definitions~\ref{defn:rvi}, \ref{defn:sequivalence} and  \ref{defn:descending} and to Figures~\ref{fig:skmoves}, \ref{fig:desc_triv} and~\ref{fig:rrkf}.

\begin{figure}[H]

 \centering
 \includegraphics[height=3.5cm]{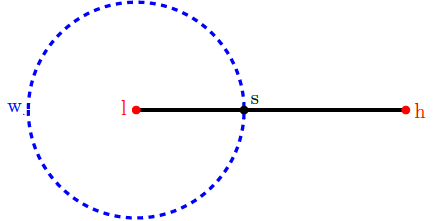}
  \caption{The abstraction of a chord diagram with winding number $w \in \Z$.}
  \label{fig:alcd}
\end{figure}

We prove Theorem~\ref{thrm:class}, which is our first main result and states that: 

{\it Every singular spherical  knotoid diagram with one singularity is singular equivalent to exactly one regular diagram. Moreover, all regular diagrams are distinct and they are classified by the winding number of the singular loop.}

A regular diagram  is a chosen representative of the singular equivalence class of a singular spherical  knotoid  with one singularity. More precisely, a regular diagram is a descending spherical knotoid diagram  with one singular crossing, containing  the minimal number of real crossings among all diagrams in its rigid vertex isotopy class, and such that the singularity is located `next to'  the leg of the knotoid, in the sense that there exist no other crossings between the leg and the singular crossing. See Definition~\ref{defn:rkd} and  Fig.~\ref{fig:rd} for an example. We recall that, when traversing a descending  knotoid diagram from leg to head, every classical crossing is encountered firstly as an overcrossing.

\begin{figure}[H]
 \centering
\includegraphics[height=6cm]{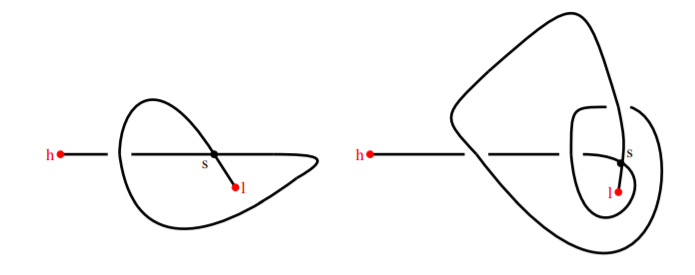}
  \caption{A regular diagram }
  \label{fig:rd}
\end{figure}

We proceed with proving Theorem~\ref{thrm:integr}, which is an integration theorem for type one invariants of spherical knotoids, and which is our second main result, stating that:

{\it Every function on linear chord diagrams which respects the one-term relation gives rise to a type-1 knotoid invariant.}

As a corollary to Theorems~\ref{thrm:class} and~\ref{thrm:integr}  we have that there are non-trivial type-1 invariants for spherical knotoids. Moreover, there are non-trivial type-1 invariants  also for planar knotoids, since a spherical knotoid corresponds uniquely to a planar knotoid whose head lies in the outer region of $\R^2$ (minus the knotoid), see Lemma~\ref{lem:trsf}.

We further define in Section~\ref{s:sh} the singular height of a knotoid $k$ to be the minimal height over all singular knotoid diagrams in the singular equivalence class of $k$.
We show that the regular diagrams realize the singular height and we then deal with the question whether we can find the singular height of a  singular knotoid just by looking at  the corresponding chord diagram. This is straightforward for spherical knotoids but it is tricky for planar so we conjecture the formula that intuitively holds, leaving it for future work.

 We conclude the paper by presenting examples of type-1 invariants for  knotoids. We show that the affine index polynomial, defined by L.H. Kauffman for virtual knots~\cite{doi:10.1142/S0218216513400075}, after  generalizing and reformulating an invariant of A. Henrich~\cite{Henrich}, and reformulated for knotoids in \cite{GUGUMCU2017186}, is a type-1 invariant (in analogy to Henrich's result for virtual knots). We  also define  a new  invariant $\bar{v}$ that is winding number detector, since it has distinct evaluation for all regular diagrams, using a good application of the integration theorem. The invariant  $\bar{v}$ is a universal type-1 invariant for spherical knotoids and it is related to the Gluing invariant of Henrich. We further reformulate the affine index polynomial using weight assignments at linear chord diagrams and prove that it is less general than $\bar{v}$.
 
 We also consider higher order Vassiliev invariants for spherical knotoids via the Jones polynomial, the Kauffman bracket polynomial and the Turaev extended bracket polynomial.

The outline of the paper is as follows. In Section~\ref{s:bg} we recall the basic definitions and results from the theory of knotoids and we give examples of some strong knotoid invariants. We also recall briefly the basics of virtual knot theory.  Then, we summarize the theory of finite type invariants and chord diagrams and recall the specific formulas for the Vassiliev invariants arising from the Jones polynomial. In Section~\ref{s:sk} we give the definition of a singular knotoid and present different types of closures for knotoids. We continue with Section~\ref{s:fti} in which we  define finite type invariants for knotoids in two ways: using the closures to induce invariants of knotoids from invariants of knots, and by defining finite type invariants directly on knotoids. In Section~\ref{s:lcd} we extend the classical chord diagrams to linear chord diagrams, which correspond to singular equivalence classes of singular knotoids. The main results are in Section~\ref{s:toi} in which we describe the notion of a regular knotoid diagram and with this we classify singular knotoid diagrams with one singularity in Theorem~\ref{thrm:class}. Then, with Theorem~\ref{thrm:integr} we demonstrate a way to produce various type-1 knotoid invariants using functions on linear chord diagrams. Finally, Section~\ref{s:exs} includes various examples of type-1 invariants, while in Section~\ref{s:sh} we discuss the singular height and in  Section~\ref{s:hoti} we show that the Turaev extended bracket gives rise to  higher type invariants.

\section{Background} \label{s:bg}

In this section we recall briefly the notions and tools used throughout the paper.

\subsection{Knotoids}\label{sub:knotoids}

In this subsection we shall recall some basic notions from the theory of knotoids. The theory was introduced by V.~Turaev in 2011 \cite{turaev2012knotoids}, where he showed that the theory of spherical knotoids extends faithfully  classical knot theory.  He also introduced several invariants for knotoids, such as the bracket and the extended bracket polymonial, while in \cite{GUGUMCU2017186} Gügümcü and Kauffman defined some more.  
  The theory of knotoids has caught the attention of
the mathematical community in the last years, enriching the theory with many more interesting results and finding applications to other aspects of science, \cite{Bartholomew, Goundaroulis_2017, Goundaroulisprotein, Dorier_2018, Gugumcu2016ASO, Braidoids, RailKnotoids2019, Barbensi, Gugumcu2018BiquandleCI, Adams_Henrich, Kutluay2020WindingHO}.

Let $\Sigma$ be a path connected, oriented $2$-manifold without boundary smoothly embedded in $\R^3$. A \textit{knotoid diagram} is a generic immersion $\gamma$ of the interval in a surface $\Sigma$ whose singularities are only finitely many transversal double points, which are endowed with over- or undercrossing data. These double points are called crossings. $\gamma(0)$ is called the leg and $\gamma(1)$ is called the head of the knotoid diagram, while both comprise the endpoints of the knotoid diagram. We will actually call a knotoid diagram its image in $\Sigma$ rather than the immersion itself. Similarly,  a {\it multi-knotoid diagram} is a generic immersion of a disjoint union of the interval and some copies of the circle, and an isotopy class of a  multi-knotoid diagram is a {\it multi-knotoid}.
For our purposes in this paper the oriented manifold will be mostly the $2$-sphere or the plane. 

Let K be a knotoid diagram with $n$ crossings. By ignoring the over/under information at each crossing of the diagram $K$ and regarding crossings as vertices, we obtain a connected planar graph with $n+2$ vertices, $n$ of which correspond to the crossings and two of which correspond to the endpoints of $K$. This graph is called the \textit{underlying graph} of the
knotoid diagram K. Here we understand that one can view every crossing as a $4$-valent vertex and the endpoints as $1$-valent vertices. Note that the underlying graph divides $\Sigma$ into $n+1$ local regions.

In classical knot theory, ambient isotopy is generated  combinatorially by three local moves on diagrams, the {\it Reidemeister moves}, together with   planar or disc  isotopy. Note that in the knotoid theory an  isotopy may displace the endpoints, but may not pull a strand adjacent to an endpoint over or under a transversal strand. This justifies the notion of the {\it forbidden moves} $\Omega_+, \Omega_-$,  illustrated in Fig.~\ref{fig:fm}.

Two knotoid diagrams $K_1, K_2$ are called \textit{isotopic} if they differ by disc isotopies of $\Sigma$  and the Reidemeister moves. An equivalence class of a knotoid diagram is called a {\it knotoid}. See Fig.~\ref{fig:iso}.

The set of knotoids in the surface $\Sigma$ is denoted by $K(\Sigma)$.

\begin{figure}[H]
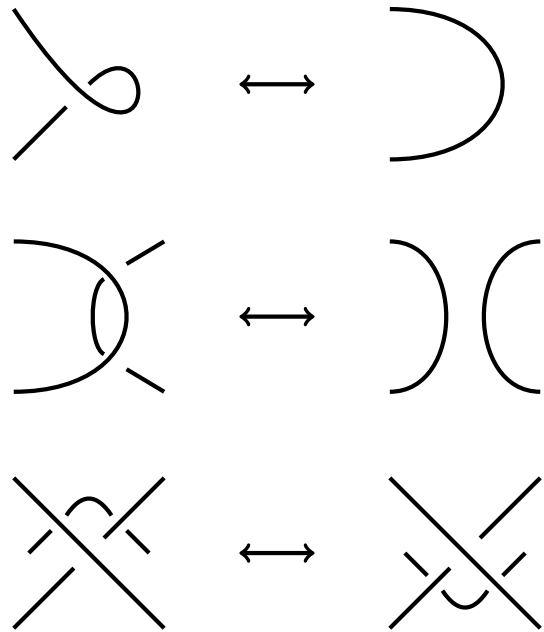


\centering

 \RONE

 \RTWO
 \vspace{30pt}
 
 \RTHREE 
 
 \caption{The Reidemeister moves}
 \label{fig:iso} 
\end{figure}

\begin{figure}[H]

\centering
\begin{tikzpicture}

    \draw[color=black, ultra thick,] (0,1.5) -- (0,-1.5);
      \draw[color=black, ultra thick,] (-1,0) -- (-0.3,0);
    
     \draw[color=black, ultra thick,<->] (2,0) -- (3,0);
     
     \draw[color=black, ultra thick,] (2.5,0) -- (2.5,0) node[anchor=south] {$\Omega_+$};

     \filldraw [red] (-0.3,0) circle (2pt) node[anchor=west] {};

     \draw[color=black, ultra thick,] (5.5,1.5) -- (5.5,0.2);
     \draw[color=black, ultra thick,] (5.5,-0.2) -- (5.5,-1.5);
      \draw[color=black, ultra thick,] (4.5,0) -- (6.5,0);
      
       \filldraw [red] (6.5,0) circle (2pt) node[anchor=north] {};
       
       \draw[color=black, ultra thick,<->] (-3,0) -- (-2,0);
       
       \draw[color=black, ultra thick,] (-2.5,0) -- (-2.5,0) node[anchor=south] {$\Omega_-$};
       
        \draw[color=black, ultra thick,] (-5.5,1.5) -- (-5.5,-1.5);
        \draw[color=black, ultra thick,] (-4.5,0) -- (-5.3,0);
        \draw[color=black, ultra thick,] (-5.7,0) -- (-6.5,0);

     \filldraw [red] (-4.5,0) circle (2pt) node[anchor=north west] {};

 \end{tikzpicture}
 \caption{The forbidden moves for knotoids}
 \label{fig:fm} 
\end{figure}
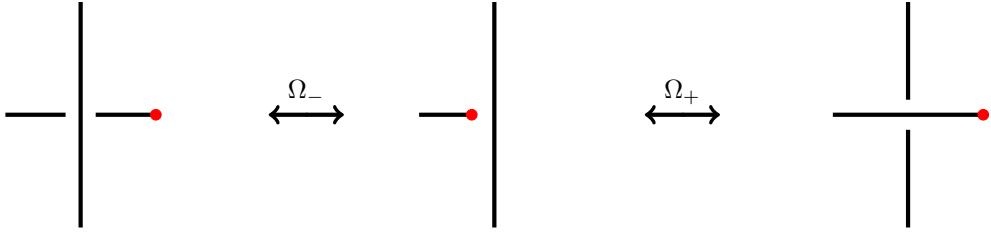

Note that isotopy of spherical knotoids corresponds to isotopy of an $(1,1)$-tangle in the annulus with the two ends attached to the two boundary components and allowed to slide along them (where the annulus is formed by removing the two points $l,h$ in $S^2$). 

 The {\it height} (or complexity) of a knotoid diagram $K$ is the minimum number  of  intersections that are created between  $K$  and a  simple arc, $\alpha$, joining the leg and the head of  $K$, such that $\alpha$ crosses $K$  transversely and only in finitely many double points   \cite{turaev2012knotoids,Gugumcu2016ASO}. Such an arc is a  {\it shortcut} for the knotoid diagram  $K$. 
 Note that a shortcut creates only transversal double points (and not triple points) so it is disjoint from the crossings.
  Furthermore, the {\it height} of a knotoid $k$ is the least number $h$, such that there exists a diagram of $k$ having height $h$.
 
 The spherical knotoids of zero height are the {\it knot-type knotoids} and their isotopy classes correspond bijectively to the classical knot types. The spherical knotoids of height greater than zero are called {\it pure knotoids}.

Given two spherical knotoid diagrams $K_1$, $K_2$ of knotoids $k_1, k_2 \in K(S^2)$ one can form the {\it product} $k_1 \cdot k_2$ by deforming both $K_1,K_2$ such that their heads are in a small neighbourhood of the point at infinity. This is achievable using isotopy in the sphere. Then we attach a copy of $K_2$ to the head of $K_1$ in a small neighborhood of the latter. Note that this can be done without ambiguity since the copy of $K_2$, being a compact subset of $S^2$, can be reduced by a sufficiently small scale, so as not to intersect $K_1$ at any point.
Here we see $S^2$ as the one-point compactification of $\R^2$, adding the point denoted by $\infty$.

Furthermore, given a knot $\kappa$ and a knotoid $k$ one can define the {\it product} $k \cdot \kappa$ in the following manner. Present $\kappa$ by an oriented knot diagram $D$ in $S^2$ and pick a small open arc $\alpha \subset{D}$ disjoint from the crossings. Then $K = D-\alpha$ is a knotoid diagram in $S^2$. This knotoid diagram is unique up to isotopy. Then take a knotoid diagram $K^{\prime}$ of $k$, as described above and multiply it with $K$. That is, $k \cdot \kappa :=  K^{\prime} \cdot K$.

One thing that is worth mentioning now for our purpose, is a special way to represent planar knotoids, which is due to   N. Gügümcü and L.H. Kauffman  \cite{GUGUMCU2017186}. Form a knot in the thickened surface $\Sigma \times F$ from the knotoid diagram in the surface $\Sigma$.
 For $K(\R^2)$, identify the plane of the planar knotoid $K$ with $\R^2 \times \{0\} \subset \R^3$. $K$ can be embedded into $\R^3$ by pushing the overpasses of the diagram into the upper half-space and the underpasses into the lower half-space in the vertical direction. This  creates an embedding of $[0,1]$ in $\R^3$, which can be viewed as a lifting of the knotoid diagram $K$ in $\R^3$. 
 The leg and the head of the diagram are attached to the two lines, $\{l\}\times\R$ and $\{h\}\times\R$ that pass through the leg and the head, respectively and are perpendicular to the plane of the diagram. Moving the endpoints of $K$ along these lines gives rise to open oriented curves embedded in $\R^3$ with two endpoints lying on each line. Such a curve is said to be the rail lifting of a knotoid diagram. An example of a knotoid $K_2$ and its rail lifting $K_1$ is illustrated in Fig.\ref{fig:rail}.
 \begin{figure}[H]
     \centering
     \includegraphics[height=10cm]{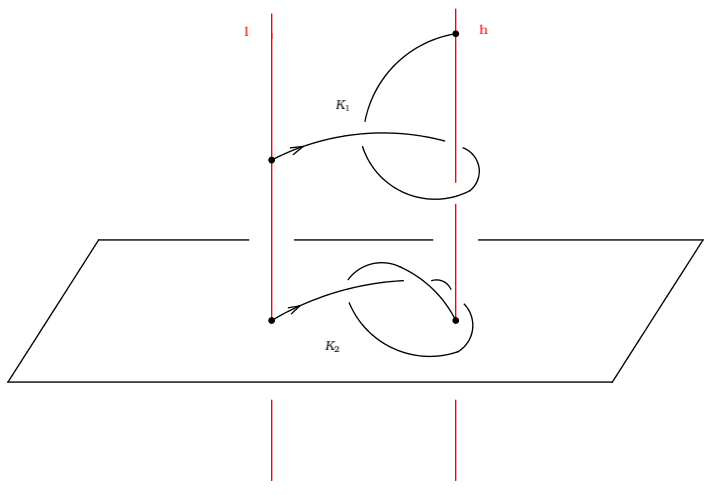}
     \caption{Example of a knotoid and its rail lifting}
     \label{fig:rail}
 \end{figure}

 Two smooth open oriented curves embedded in $\R^3$ with the endpoints
that are attached to two special lines, are said to be \textit{line isotopic} (or  \textit{rail isotopic} \cite{RailKnotoids2019}) if there is a smooth ambient isotopy of the pair $(\R^3 / \{t\times \R, h \times \R\}, t\times \R \cup h \times \R)$, taking
one curve to the other curve in the complement of the lines, taking endpoints to endpoints, and lines to lines.

In \cite{GUGUMCU2017186} Gügümcü and Kauffman prove that there is a one-to-one correspondence between the set of knotoids
in $\R^2$ and the set of line-isotopy classes of smooth open oriented curves in $\R^3$, which have their two ends attached to the two parallel lines. Furthermore, in \cite{RailKnotoids2019} D. Kodokostas and S. Lambropoulou develop the theory of {\it rail knotoids}, which are projections of line curves in the plane of the two parallel lines (the `rails').

\subsection{A digression on virtual knots}\label{sub:virtual}

Virtual knot theory was introduced by L.H.~Kauffman in 1998 \cite{KAUFFMAN1999}. We recall that a {\it virtual knot diagram} is an immersion of $S^1$ in the plane, containing finitely many double points, some of which are classical crossings and some are virtual crossings. The  virtual crossings can be viewed as permutation  crossings with no `under' or `over' specification. Virtual isotopy comprises the planar isotopy, the Reidemeister moves among classical crossings and the extra isotopy moves involving also virtual crossings, as illustrated in Fig.~\ref{fig:virtual}.  We note that the isotopy moves involving also virtual crossings are special cases of the so-called {\it detour move}, whereby, any sequence of virtual crossings only can be cut out and replaced by a connection elsewhere in the diagram.
 A {\it virtual knot} is a virtual isotopy class of virtual knot diagrams. We note that  Reidemeister moves of type III, involving two classical crossings and one virtual are not allowed in this theory.

Virtual knotoid diagrams are defined in \cite{GUGUMCU2017186} to be the knotoid diagrams in $S^2$ with an extra combinatorial structure called virtual crossings, viewed again as permutation  crossings. Virtual knotoid is an equivalence class of virtual knotoid diagrams up to virtual isotopy away form the endpoints.

\begin{figure}[H]
    \centering
\includegraphics[scale=0.9]{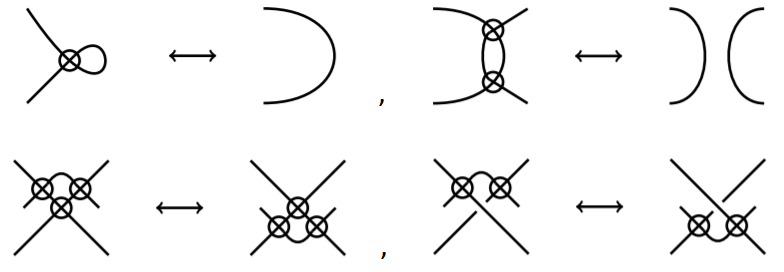}
    \caption{Isotopy moves involving virtual crossings.  }
    \label{fig:virtual}
\end{figure}

\subsection{Finite type invariants and chord diagrams}\label{sub:fti_knots}

A \textit{singular knot} is a mapping of $S^1$ in $\R^3$ that fails to be an embedding only in finitely many points where we have only transversal self-intersections, the {\it singular crossings}. Singular knots, like classical knots, are best understood by their plane projections. A \textit{singular knot diagram} is a plane curve whose only singularities are transversal double points. These points are endowed with information over, under or singular. A singular knot diagram can be viewed as a regular projection of a singular knot in some vertical direction, with respect to a chosen plane. Two singular knots are said to be {\it rigid vertex isotopic} if any two diagrams of theirs differ by a fine sequence of disc isotopies, the Reidemeister moves for classical knots, and the rigid vertex isotopy moves involving singular crossings. See Figs.~\ref{fig:iso} and \ref{fig:skmoves}.  See for example \cite{Vassiliev1990}, \cite{BARNATAN1995423}, \cite{KAUFFMAN2004DIAGR}.

\begin{figure}[H]
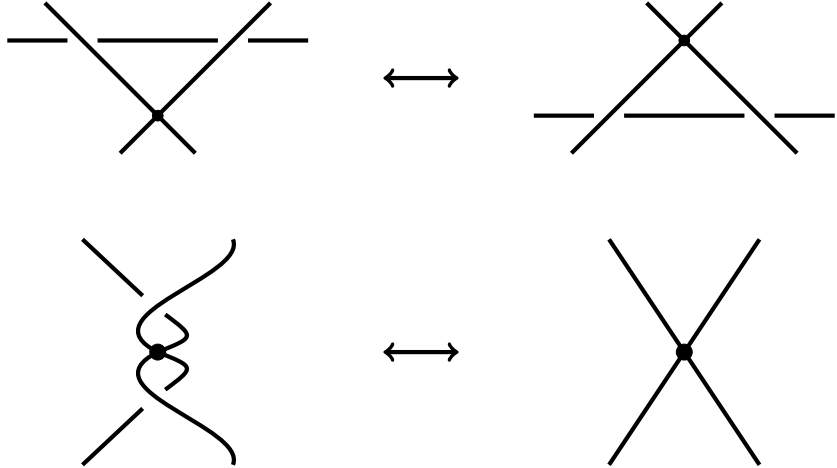


\centering

\SRTHREE 
\vspace{30pt}
\SFLIP
 
 \caption{Rigid vertex isotopy moves for singular crossings}
\label{fig:skmoves}
\end{figure}

Any knot invariant $v$ can be extended to an invariant of singular knots using the \textit{Vassiliev skein relation}:
\begin{equation} \label{eqn:vskein}
v(\SC) = v(\PC) - v(\NC)
\end{equation} 
The small diagrams in the relation denote diagrams that are identical except for the regions in which they differ only as indicated in the small diagrams.  Using this relation successively, one can extend $v$ to singular knots with an arbitrary number of singular crossings. While there are many choices when taking and resolving a sequence of singular crossings,  in fact the complete resolution yields the alternating sum
$$
\sum_{\e_1=\pm 1, \dots ,\e_n=\pm 1}(-1)^{|\e|}v(K_{\e_1 \e_2 \dots \e_n}) 
$$

where $|\e|$ is the number of $-1$'s in the sequence of $\e_i$ and $K_{\e_1 \e_2 \dots \e_n}$ the knot obtained by positive/negative resolution of each singular crossing of $K$.

A knot invariant $v$ is said to be of \textit{finite type $k$} if there exists a $k \in \mathbb{N}$ such that $v$ vanishes on every knot with more than $k$ singular crossings. We say that $v$ is of order (or type) $\le k$. See \cite{Vassiliev1990}, \cite{CHMUTOV2011SURVEY}.

Rigid vertex isotopy together with crossing switches generate an equivalence relation, the {\it singular equivalence}, which is used for the so-called `top row' diagrams.

\smallbreak

\begin{rem} \rm

The technique of (classical) crossing switches can be applied to any finite type invariant of order $k$ when applied to a `top row' diagram (that is, a singular knot with precisely $k$ singular crossings). Namely, one can use Eq.~\ref{eqn:vskein} for switching any real crossing at will.
As a result, the value of a finite type invariant of order $k$ depends only on the nodal structure of the singularities (the underlying 4-valent graph) and not on the embedding of the graph in space.  
\end{rem}

A {\it chord diagram of order $n$} comprises a circle with $n$ chords, which encode the number and the sequence of double points along a singular knot.  For illustrations see Fig.~\ref{fig:chord_diagrams}. 
A Vassiliev invariant of order $k$ gives rise to a function on chord diagrams with $k$ chords. The conditions that a function on chord diagrams should satisfy in order to come from a Vassiliev invariant, are  the so-called one-term and four-term relations. The vector space spanned by chord diagrams modulo these relations has the structure of a Hopf algebra. This Hopf algebra turns out to be dual to the algebra of the Vassiliev invariants and this isomorphism is the famous Vassiliev-Kontsevich theorem \cite{Kontsevich1993}. For detailed context see \cite{CHMUTOV2011SURVEY}. There are some interesting questions concerning this fact, including if there exists a similar structure for knotoids.

\begin{figure}[H]
    \centering
    \includegraphics[height=2.2cm]{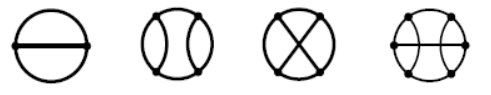}
    \caption{Chord diagrams}
    \label{fig:chord_diagrams}
\end{figure}

In the original work of V.~Vassiliev \cite{Vassiliev1990}, finite type invariants correspond to the zero-dimensional classes of a special spectral sequence. In fact, the space of knots is understood as the complement of the space of mappings of $S^1$ in $\R^3$ that fail in some way to be  embeddings (the so-called discriminant). Knot invariants are just locally constant functions in the space of knots and Vassiliev's formalism was trying to show that there exists such a spectral sequence that converges to the cohomology of the space of knots. These definitions and calculations have been very much simplified by the works of Bar-Natan \cite{BARNATAN1995423, BARNATAN2002}, Birman \& Lin \cite{Birman1993}, Stanford \cite{STANFORD19961027, Stanford2004}, Sawollek \cite{Sawollek} and Polyak \& Viro \cite{PolyakViro1994, Goussarov1998FinitetypeIO},  as we are interested only in this zero-class, and as  the use of the chord diagrams made a big part of the theory combinatorial.

\section{Singular knotoids and types of closures } \label{s:sk}

In this section we define the notion of singular knotoid in any connected, oriented surface $\Sigma$.

\subsection{Singular knotoid diagrams and rigid vertex isotopy}

\begin{defn}\rm\label{defn:singular_knotoid_diagram}
A \textit{singular knotoid diagram} is a knotoid diagram with some (finite) undeclared double points which are transversal  (i.e. the two corresponding velocity vectors of the curve are linearly independent). See Fig.~\ref{fig:skd} for some examples.
\end{defn}

\begin{figure}[H]

 \centering
\includegraphics[height=3.5cm]{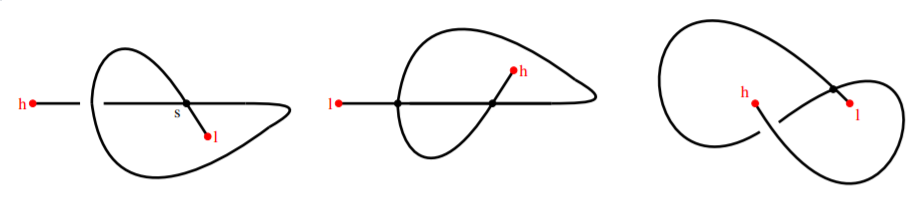}

  \caption{Singular knotoids}
\label{fig:skd}  
\end{figure}

\begin{defn}\rm\label{defn:rvi}
A {\it rigid vertex isotopy for singular knotoid diagrams} is generated by locally planar isotopy and the Reidemeister moves for classical knotoids (recall Fig.~\ref{fig:iso}), extended by the rigid vertex isotopy moves involving singular crossings (recall Fig.~\ref{fig:skmoves}).  An isotopy class of singular knotoid diagrams is called  {\it singular knotoid}.
\end{defn} 

Note that in the theory of singular knotoids we still have the restrictions of the forbidden moves (recall Fig.~\ref{fig:fm}).

\subsection{Types of closures}

\begin{defn} \rm 
Let $K$ be a (classical or singular)  knotoid diagram in a surface $\Sigma$. We call {\it classical closure of type `o' (resp. `u')} for $K$  the (classical or singular)  knot diagram obtained by joining the endpoints of $K$ by a shortcut such that all formed crossings are declared to be `over' (resp. `under') $K$. 
Consequently, the {\it classical closure of type `o' (resp. `u')} for (classical or singular)  knotoids in $\Sigma$ is the mapping $c^o $ (resp. $c^u $) from the set of (classical or singular) knotoids in $\Sigma$ to  the set of (classical or singular) knots in the thickened surface $\Sigma \times I$, induced by the over- (resp. under-) closure on knotoid diagrams. More precisely, for a knotoid $[K]$ (which is an equivalence class of knotoid diagrams) we define 
$$
\displaystyle{c^o([K]) := [K^o]} \quad \mbox {and} \quad \displaystyle{c^u([K]) := [K^u]}
$$
where  $K^o, K^u$ denote the (singular) knot diagrams in $\Sigma$ obtained via  the over- resp. under-closure of the knotoid diagram $K$. See top row of Fig.~\ref{fig:clos}. 
\end{defn}
As we will see the classical closure is well-defined for planar and spherical knotoids. However, in surfaces of higher genus there is a multiplicity of closures, related to non-equivalent homotopy classes.

\begin{figure}[H]

\centering

\includegraphics[height=10cm]{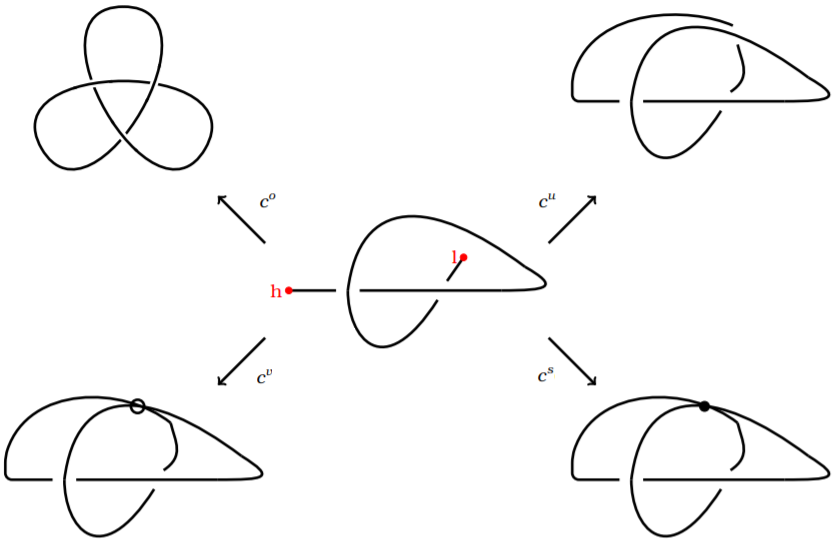}

 \caption{Types of closures: over- and under-closures (top row), virtual closure (bottom left), and singular closure of height 1 (bottom right)}
\label{fig:clos}
\end{figure}

\begin{lemma}  \label{cclosure}
The classical closure of type `u'  (resp. `o') for spherical or planar  classical knotoids (or singular knotoids) is well-defined up to isotopy (or rigid vertex isotopy).
\end{lemma}

\begin{proof} \rm

Let  $[K]$ be a spherical or planar (singular) knotoid. We will prove the lemma for the `u' closure of $[K]$ (the proof for the `o' closure is completely analogous). Namely, we want to show that the under-closure determines a unique (singular) knot $[K^u]$ up to isotopy. 

We will first show that the `u' closure does not depend on the choice of shortcut on the same diagram. Indeed, let $K$ be any diagram of $[K]$ and let  $\alpha$ and $\beta$ be two different shortcuts connecting the two endpoints of $K$, so that in all formed crossings $\alpha$ and $\beta$ pass under $K$. Then $\beta$ can be freely isotoped to  $\alpha$ by a sequence of oriented Reidemeister moves (possibly together with rigid vertex isotopy moves), since they both lie under the rest of the diagram. 

Let now $K^\prime$ be another diagram of $[K]$. Then $K$ and  $K^\prime$ differ by a sequence of oriented Reidemeister moves (possibly together with rigid vertex isotopy moves). Let further $\alpha$ and $\alpha^\prime$ be two shortcuts for $K$ and  $K^\prime$ respectively, realizing the `u' closures. Then,  since they both lie under the rest of the diagrams, and by the Reidemeister moves (and possibly the rigid vertex isotopy moves) which are also valid in the closures, we have  that the resulting (singular) knot diagrams $K^u$ and ${K^\prime}^u$ are  (rigid vertex) isotopic. Hence, we obtain a unique (singular) knot $[K^u]$ up to isotopy. 
\end{proof}

\begin{rem} \rm
Suppose that for a knotoid we have a diagram $K$ which realizes its height. One might consider taking all possible closures (using both under and over crossing) by taking a diagram that realizes the height and an arc connecting leg to head and not passing through any other crossing, but creating $c$ new crossings. Then one would declare with $2^c$ options which crossing is under and which over. Given two  isotopic knotoid diagrams   $K^1$,$K^2$  with the same height, say $2$,  there is no way to prove that the knot $K^1_{o,u}$ is knot-isotopic with the knot $K^2_{o,u}$ or $K^2_{u,o}$, where $K^n_{o,u}$ is the knot that we obtain if, starting from the leg of $K^n$ and  following the arc $\alpha$, we declare the first extra crossing to be over-crossing and the second one to be under-crossing. Respectively, we denote $K^n_{i_1,i_2,...,i_c}$  to be the closure of the knotoid $K^n$ with height $c$, such that $i_1,...,i_c \in \{o,u\}$ and if $i_k=o$ this means that the $k$th extra crossing is declared to be over, while if $i_k=u$ this means that the  $k$th extra crossing is declared to be under, $k=1,\ldots,c$. The above means that $K^n_{i_1,i_2,...,i_c}$ is not well-defined, as different shortcuts will  possibly give non isotopic knots. 

For example if we take two distinct shortcuts in the knotoid illustrated in Fig.~\ref{fig:knotoid_barth} , we will end up with different normalized bracket polynomials with the $(u,o)$-closure.

\end{rem}
\begin{figure}[H]
    \centering
    \includegraphics[height=4cm]{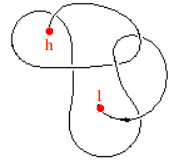}
    \caption{A knotoid that attains two distinct $(u,o)$-closures}
    \label{fig:knotoid_barth}
\end{figure}

We will now introduce the virtual closure. Recall the extra isotopy moves between virtual knot diagrams, illustrated in Fig.~\ref{fig:virtual}.

\begin{defn} \rm
We call {\it virtual closure for  knotoids } (classical, singular or virtual)  a mapping $c^v $ on the set of (classical, singular or virtual)  knotoids in a surface $\Sigma$ to  the set of virtual knots in the thickened surface $\Sigma \times I$, induced by the virtual closure on knotoid diagrams, whereby all extra crossings are declared  to be virtual when taking a shortcut for the knotoid diagram. Namely, 
$$
\displaystyle{c^v([K]) := [K^v]}
$$
where  $K^v$ denotes the virtual knot diagram in $\Sigma \times I$ obtained via  the virtual closure of a knotoid diagram $K$. See bottom left  of Fig.~\ref{fig:clos}. 
\end{defn}

We further have:

\begin{defn} \rm\label{defn:virtual_knotoid}
Two singular virtual knotoid diagrams are equivalent or {\it singular virtual isotopic} if they differ by a finite sequence of moves of the following types: planar isotopy, Reidemeister moves, detour moves for  virtual crossings, and by rigid vertex moves that involve classical and virtual crossings. Note that detour moves apply in general. That is, any sequence of virtual crossings only can be cut out and replaced by a connection elsewhere in the diagram, even when the diagram has both classical and singular crossings.
\end{defn}

\begin{lemma}
The virtual closure for spherical or planar knotoids (classical, virtual or singular) is well-defined up to virtual isotopy (and/or rigid vertex isotopy). 
\end{lemma}

\begin{proof}
Clearly, the closing arc may move freely by the detour move, which is also valid in the presence of singular crossings, so the definition does not depend on the choice of shortcut. 
Moreover, any virtual/rigid vertex isotopy move between knotoid diagrams is also a valid isotopy move for their virtual closures. 
This is also the idea in the proof of the analogous result for classical knotoids in \cite{GUGUMCU2017186}.
\end{proof}

We say that a knotoid diagram is {\it prime} if it is not a connected sum of a knot diagram with a knotoid diagram or a product of two knotoid diagrams.

A knotoid is prime if it does not possess any connected sum in its equivalence class.

A \textit{local knot} is a knot diagram inside a disc, seen as an $(1,1)$-tangle across the disc.

Note that a prime knotoid diagram has no local knotting.

For (singular) knotoids of height one, we can also define a singular closure:

\begin{defn} \rm

A \textit{singular closure} for a knotoid diagram $K$ in a surface $\Sigma$,  which realizes height, is defined by the following procedure: Take a classical closure of $K$ using a shortcut that realizes height and nodify the extra crossings. The resulting singular knot diagram in $\Sigma$ is denoted by $K^s$. See bottom right of Fig.~\ref{fig:clos}. 

Furthermore, the singular closure of a knotoid diagram that does not realize height is defined by taking the singular closure of another knotoid diagram in its isotopy class that does realize height.
\end{defn}

\begin{lemma}\label{sclosurewelldef} 
The singular closure for classical spherical or planar prime knotoids of height one is well-defined up to rigid vertex isotopy.
\end{lemma}

\begin{proof}
We will first argue that the  definition of singular closure does not depend on the choice of shortcut.  Indeed, since our diagram realizes the height one, there is at least one place where the leg is separated by the head by just one boundary arc. If there is only one such place, we are done. Suppose there are more places. We consider two neighbouring. Then the union of the associated  shortcuts makes a simple closed curve. Note that, in this case, one can choose them so as to not intersect each other. This closed curve bounds a disc, and in this disc the boundary arc is not  knotted, since this knotting would be by definition local, so the two shortcuts differ by rigid vertex isotopy.
\end{proof}

\begin{rem} \rm
Singular closure is well-defined only for height-one prime knotoids and, moreover, on diagrams that realize the height of the knotoid they represent. The requirement for prime knotoids becomes apparent in the proof of Lemma~\ref{sclosurewelldef}. To see the height-1 requirement, take, for instance, the counterexample of height two, shown in Fig.~\ref{fig:sing_ex}.

\begin{figure}[H]
    \centering
    \includegraphics[height=4cm]{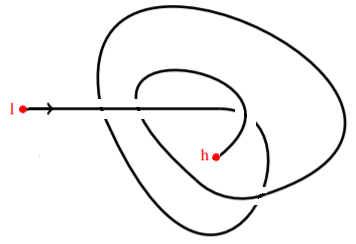}
    \caption{A knotoid diagram that admits two distinct singular closures}
    \label{fig:sing_ex}
\end{figure}

\noindent Indeed, calculating its affine index polynomial (see Subsection~\ref{sub:affine}) we obtain 
$$P_K=t^2+t^1+t^{-1}+t^{-2}-4$$
so, by \cite{GUGUMCU2017186} the knotoid is indeed of height $2$. Yet, it admits two distinct singular closures, as indicated in Fig.~\ref{fig:sing_clos}:  the closure $s_1$ illustrated by the left picture gives rise via an oriented smoothing of the singular crossings to a non-trivial link of $3$ components. However, the closure  $s_2$ illustrated by the right picture gives rise to the trefoil knot via an oriented smoothing of the singular crossings. So, the two singular closures are not isotopic and, therefore, the singular closure is not well-defined for height two.
A complete description of the oriented smoothing and that it is rigid vertex isotopy invariant can be found in \cite{KauffmanGraphs89}.

\end{rem}

\begin{figure}[H]
    \centering
    \includegraphics[height=4cm]{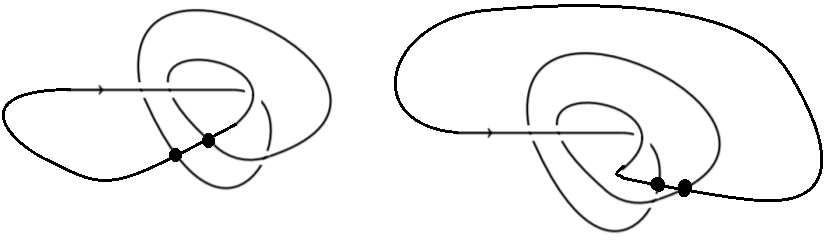}
    \caption{Two distinct singular closures}
    \label{fig:sing_clos}
\end{figure}

\begin{rem} \rm \label{rem:clos_inv}

If $k$ is a spherical or planar knotoid and $\bar{k}$ denotes any type of  closure that can be defined on $k$  (classical, virtual or singular), then this closure defines a map from knotoids to classical, virtual or singular knots, so that, by the lemmas above, the  knot $\bar{k}$ can be regarded as an invariant of the knotoid $k$. This association is not a complete invariant of knotoids, as one can easily construct examples of non-equivalent knotoids having isotopic closures. For example, there are non-trivial knotoids with trivial virtual closures. Yet, this association has been a source of invariants of knotoids by taking invariants of the closures. For example, the Jones polynomial and the affine index polynomial for knotoids are identical with the corresponding invariants of their virtual closures.
\end{rem}

\section{Finite type invariants for knotoids} \label{s:fti}

By Remark~\ref{rem:clos_inv} any finite type invariant of knots (classical or virtual) is also a finite type invariant for knotoids, which, however, is not expected to be classifying.   
In this paper we define intrinsically in terms of knotoids first order Vassiliev invariants of knotoids. These are defined in terms of weight systems and form the first step in a more general theory of Vassiliev invariants for knotoids. 
 
For the rest of the paper we  focus on knotoids  in $S^2$ (classical, singular, virtual, virtual singular). Occasionally we will be making more general definitions and comments, stating explicitly the generalization.

\subsection{Finite type invariants obtained by closures}

\begin{defn} \rm
Let $v$ be any $R$-valued finite type invariant for classical knots, where $R$ is a commutative ring. We call {\it classical closure finite type invariant for (classical) knotoids related to $v$} a mapping $v^c $ from the set of (classical) knotoid diagrams to $ R^2 $, defined in the following way: 
$$
\displaystyle{v^c(K) := (v(K^u), v(K^o))}
$$
where  $K^u, K^o$ are respectively the knot diagrams obtained via under- and over-closure, respectively, of a knotoid diagram $K$. 
\end{defn} 
 
What we want to prove is that given a finite type invariant for knots, $v$, induces an isotopy invariant for spherical knotoids using the classical closure.
For  isotopic knotoids $K_1, K_2$, and for any (knot) finite type invariant $v$, we want $v^c(K_1)=v^c(K_2)$, or equivalently we want $v(K^u_1)=v(K^u_2)$ and $v(K^o_1)=v(K^o_2)$ but this is true because $v$ is an isotopy invariant for knots and a knotoid  $K \subset{S^2}$ determines (via the under-closure) a unique knot $K_u \subset{S^3}$ (resp.~$K_o \subset{S^3}$) up to isotopy from Lemma~\ref{cclosure}.

\begin{defn} \rm
Let $v$ be any finite type invariant for virtual knots, valued in a commutative ring $R$. We call {\it virtual closure finite type invariant for knotoids related to $v$} a mapping $v^v $ from the set of (virtual) knotoid diagrams to   $R$, by computing the invariant on the resulting virtual knot. Namely,
$$
v^v(K) := v(K^v) 
$$
where $K$ is a   (virtual)  knotoid diagram and $K^v$ is the virtual knot diagram obtained by the virtual closure. 
\end{defn}

For finite type invariants of virtual knots, see for example \cite{KAUFFMAN1999}.

\begin{defn} \rm
Let $v$ be any finite type invariant for singular knots valued in a commutative ring $R$. We call {\it singular closure finite type invariant for prime knotoids of height one related to $v$} a mapping $v^s$ from the set of (singular) prime knotoid diagrams realizing the height to $R$, by computing the invariant on the resulting singular knot. Namely,
$$
v^s(K) := v(K^s) 
$$
where $K$ is a  (singular) knotoid diagram realizing the height and $K^s$ is the singular knot diagram obtained by the singular closure. One could  then say that this knotoid invariant is of order (or type) $\le n$ if it vanishes on all singular knotoids with more than $n-1$ singular crossings.
\end{defn}

\subsection{Finite type invariants defined directly on  knotoids}

\begin{defn}\rm\label{defn:sequivalence}
Two singular knotoid diagrams $K_1,K_2$ in a connected, oriented surface $\Sigma$ are said to be \textit{singular equivalent} if one can be deformed to the other by rigid vertex isotopy moves (recall Definition~\ref{defn:rvi}, Figs.~\ref{fig:iso} and \ref{fig:skmoves}) and real crossing switches. We will use the symbol $K_1\sim K_2$ for singular equivalence of the knotoids $K_1, K_2$.
\end{defn}

The principal idea of the combinatorial approach to the theory of finite type invariants is to extend a knotoid invariant $v$ to singular knotoids, by making use of the  Vassiliev skein relation, recall Eq.~\ref{eqn:vskein}. 

So we define: 

\begin{defn} \rm\label{defn:fti}
A knotoid invariant, $v$, is said to be a {\em finite type invariant of order (or type) $\le n$} if its extension on singular knotoids vanishes on all singular knotoids with more than $n$ singular crossings. Furthermore, $v$ is said to be of order (or type) $n$ if it is of order $\le n$ and not of order $\le n-1$.
\end{defn}

Let $v$ be a Vassiliev invariant of type $n$. The top row of $v$ is the set of values that $v$ assigns to the set of singular knots with precisely $n$ singular crossings (one can view them as abstract 4-valent graphs with $n$ nodes)

For any finite type invariant $v$  we consider the top row singular knotoid diagrams up to singular equivalence. Now, as in the classical case, any  finite type invariant remains unchanged under crossing switches on top row diagrams, due to the Vassiliev skein relation. Indeed, if $K$ is a top row diagram and has a positive (resp. negative) crossing we can switch this crossing by invoking the Vassiliev relation
$$
v(\PC) = v(\SC)+v(\NC). 
$$
But $v(\SC)=0$ because this diagram has $n+1$ singular crossings, since  $K$ is a top row diagram. Hence $v(\PC)=v(\NC)$. 

From the above, a top row singular knotoid may be represented by an appropriately chosen `regular' representative of its singular equivalence class.

For examples of finite type invariants defined directly on knotoids we prompt the reader to Sections~\ref{s:exs} and~\ref{s:hoti}.

\section{Linear chord diagrams } \label{s:lcd}

In the theory of finite type invariants for classical knots there is a natural way to encode the order of the singularities in a circular diagram  and then join each pair of points of the circle with a simple interior arc (obtaining the chord diagram).

For further work, we intend to examine the results of Petit \cite{Petit}, who generalized the finite type invariants of Henrich \cite{Henrich}, including her universal invariant, and based his generalization on the concept of the matrix for a virtual string, first introduced by Turaev \cite{Virtualstrings, Virtualstringscobordisms} and Silver \& Williams~\cite{openvirtualstrings}.

Since a  knotoid diagram is an immersion of $[0,1]$  in some surface, we must think of chord diagrams as `chords' joining points of an open-ended smooth curve in the surface, that is, a connected, oriented $1$-manifold with boundary two endpoints, which correspond to the two endpoints of the singular knotoid diagram. Each chord corresponds to precisely one singular crossing of a singular knotoid diagram, and the chords respect the forbidden moves.  With these restrictions there are many non-isotopic ways to join two points of the interval with a simple arc. This is not a technicality but a first attempt to understand the complexity of the problem since chord diagrams are closely connected with singular equivalence classes of knotoids.

\subsection{Linear chord diagrams in \texorpdfstring{$ S^2 $}{.}}

\begin{defn}\rm\label{defn:lcd}

A\textit{ linear chord diagram} of order (or degree) $n$  is an oriented closed interval, with its endpoints marked with `$h$'  and `$l$' (for `head' and `leg'). It is equipped with a distinguished set of $n$ disjoint pairs of distinct points which are connected with immersed unknotted arcs in the sphere in general position, the `chords'. We consider linear chord diagrams up to orientation preserving diffeomorphisms of the interval, and up to isotopy of the immersed arcs which fixes the ends. All the intersections between the chords themselves and between the chords and the interval are declared to be flat. The set of all linear chord diagrams of order $n$ will be denoted by $\displaystyle{A^{l}_n}$. 
\end{defn}

Fig.~\ref{fig:lcd} illustrates examples of linear chord diagrams. The symbol `$s$' stands for `singular crossing'. It is clear that there are many ways that two points can be connected with an immersed arc even though being considered up to arc isotopy and diffeomorphisms of the interval. For example, consider two points connected by two different arcs, the first one wrapping around the head with the counterclockwise orientation and the second wrapping around the leg  with the counterclockwise orientation. See Fig.\ref{fig:lcd_a_b}. These two arcs belong to different homotopy classes of the annulus (that is, the sphere with the points $h$ and $l$ removed).  
So one thinks that even the set  $\displaystyle{A^{l}_1}$ is not trivial. 

For the time we will restrict ourselves to the case $n=1$. We will need some lemmas in succession to classify all different linear chord diagrams of order $1$. First we consider two generators $a,b$: Generator $a$ represents a wrapping of the chord around the leg with the counterclockwise orientation,  meaning that, if we have a disjoint pair of points in the interval starting with the point closer to the leg with destination the point closer to the head and following the immersed arc, we go exactly once around the leg. Generator $b$ represents a wrapping of the chord around the head respectively. For the clockwise orientations we will use the symbols $a^{-1}$ and  $b^{-1}$ respectively.

\begin{figure}[H]
    \centering
    \includegraphics[height=2cm]{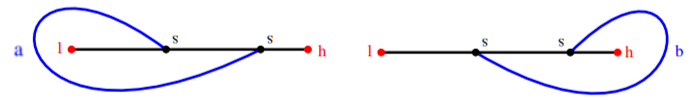}
    \caption{The generators $a,b$}
    \label{fig:lcd_a_b}
\end{figure}

We now define a multiplication operation on the set of linear chord diagrams  $A_{1}^l$ . 
Multiplication $x\cdot y$ means that we place all four ends in the same interval such that the two involving $x$ are both before the ones involving $y$, with respect to the natural orientation of the interval, and then we do the following procedure. All the intersections between $x$ and $y$ are declared to be flat. We then  take a neighbourhood of the sub-interval starting from the destination of $x$ and ending at the start of $y$. Inside there we connect the destination of $x$ with the start of $y$ in the unique way that the resulting curve is transversal to the interval precisely at both points. For an illustration see Fig.~\ref{fig:a*b}. The result is a pair of points of the interval together with an immersed curve joining them, whose interior meets the interval at every point transversely. 

The operation multiplication is well-defined up to flat isotopy and diffeomerphisms of the interval. It is also clearly associative. 
Note, further, that there are two trivial chord diagrams, in which the chord does not wind around the leg or the head: the one that the chord lies below the interval and the one that the chord lies above the interval, denoted by $1_u, 1_o$ respectively. 
Of course they are flat isotopic, hence $1_u = 1_o := 1$.  They both correspond to knot-type singular knotoids and so they can be identified with the classical chord diagram with one chord. 

The following result justifies the choice of the notation for $a^{-1}$ and $b^{-1}$.

\begin{lemma} \label{lem:a*a}
$a\cdot a^{-1} = 1\quad b\cdot b^{-1} = 1$.
\end{lemma}

\begin{proof}
The statement is an immediate consequence of the definition of the multiplication  with $x=a$ and $y=a^{-1}$, as illustrated in Fig.~\ref{fig:a*a^-1}. It is clear that $a \cdot a^{-1}$ is flat isotopic to the trivial chord diagram.

\begin{figure}[H]

 \centering
 \includegraphics[height=4.5cm]{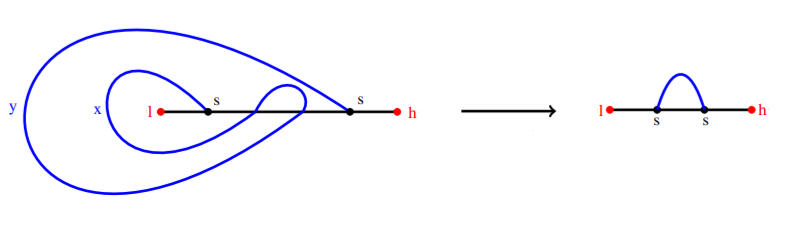}
 
  \caption{$a\cdot a^{-1}=1$}
  \label{fig:a*a^-1}
\end{figure}

\end{proof}

\begin{lemma} \label{lem:a*b}

$a\cdot b = 1 \quad b\cdot a=1$ in $S^2$.

\end{lemma}

\begin{proof}
Now the idea is this. Take $x=a$ and $y=b$. Connect the endpoints as the definition suggests such that both arcs cross the interval transversely. Then Fig.~\ref{fig:a*b} shows that in $S^2$ $a \cdot b$ is trivial.

\begin{figure}[H]

 \centering
\includegraphics[height=4.5cm]{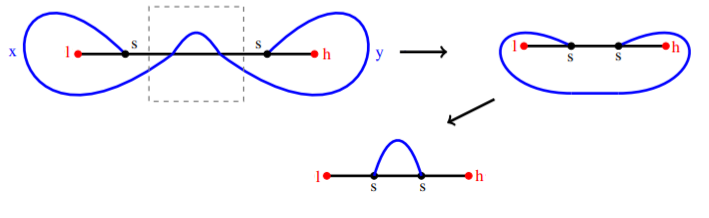}
  \caption{$a\cdot b = 1$}
  \label{fig:a*b}
\end{figure}

\end{proof}

\begin{cor}
From the above it follows that the set $A_{1}^l$ of spherical linear chord diagrams equipped with the multiplication forms a group. Lemma~\ref{lem:a*b} leaves us with only one generator in our construction, say $a$. So we can assume that every spherical chord diagram  has its head in the outer region.
We make this reference for future use, because it is fundamental and may well apply to further work on Vassiliev invariants for knotoids, when the weight systems are more complex.
\end{cor}

From the above, it is clear that we may rethink of a linear chord diagram with one chord, representing a singular knotoid with exactly one singularity, as a closed interval with a loop winding (only) around the leg $w$ times. The integer $w$ is the {\it winding number} of the loop. So, the two points $s$ can collapse into a single point and the loop can be abstracted to a  circle with the indication $w$.  See Fig.~\ref{fig:alcd}.

Furthermore, there is clearly no torsion in the cyclic group $A_{1}^l$. So, from the above we have proven the following:

\begin{prop} 
In $S^2$ for any linear chord diagram $C \in A_{1}^l$ there exists some $w \in \Z$ such that $C=a^w$. As a group, $A_{1}^l$ is infinite cyclic.
\end{prop}

Note that in $\R^2$ this technique wouldn't work, since, for example, the second diagram in Fig.~\ref{fig:a*b} is locked and non-isotopic to the third one. So, if one tries to extend this theory to $\R^2$ or a surface $\Sigma$, one would have to deal with more complexities.

\subsection{From a singular knotoid  in \texorpdfstring{$S^2$}{} to its chord diagram}

Given a singular spherical knotoid diagram with one or more singular crossings we will correspond to it unambiguously a linear chord diagram.  

If we allowed to join two points without winding around $l$ or $h$ we would have the same chord diagram for non singular equivalent knotoids. For example, the knotoids shown in Fig.~\ref{fig:nonisoknotoids} are clearly non singular equivalent, yet, with the above convention they would both correspond to the trivial chord diagram (illustrated in the second row of Fig.~\ref{fig:a*b}).

\begin{figure}[H]
 \centering
  \includegraphics[height=3cm]{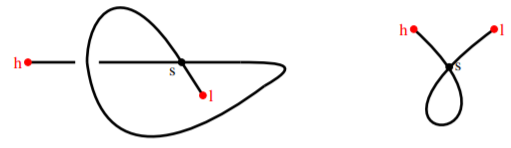}
  \caption{Two non singular equivalent knotoid diagrams}
  \label{fig:nonisoknotoids}
\end{figure}

The first observation to address this problem is that a knotoid diagram which winds once around the leg counter-clockwise from the first time we traverse the singular crossing until the second time, should correspond to a chord diagram winding once around the leg in the  counter-clockwise sense, as in Fig.~\ref{fig:lcd1}

\begin{figure}[H]
 \centering
 \includegraphics[height=3cm]{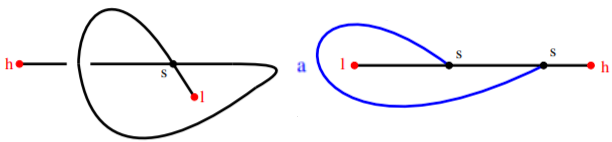}
  \caption{A knotoid diagram with its corresponding chord diagram}
  \label{fig:lcd1}
\end{figure}

We also note that there is no point in considering paths that wind around the head since we know that this would correspond to a path winding around the leg the same amount of times but with reversed orientations.

With these observations in mind, we propose the following algorithm for corresponding a linear chord diagram to a singular spherical knotoid diagram $K$ of order 1. Before doing that, we give the definition of the singular loop.

\begin{defn}\label{defn:sing_loop}\rm
The {\it singular loop} of a singular knotoid diagram (resp. of a singular knotoid) with one singular crossing is the closed loop formed in the diagram (resp. the knotoid) starting and ending at the singular crossing.

\end{defn}
\smallbreak

\noindent {\bf Algorithm for order 1 singular spherical knotoid diagrams}
\begin{itemize}
    \item  Split $K$ into $3$ paths: The path joining the leg to the singular point ($l \rightarrow s$), the \it singular loop ($s \rightarrow s$), and the path joining the singular loop  with the head ($s \rightarrow h$). See Fig.~\ref{fig:LCD_CONSTR} top left.
   \item   Erase the paths $l \rightarrow s$ and $s \rightarrow h$, but keep track of the points $l,h$. See Fig.~\ref{fig:LCD_CONSTR} top right.
    \item  What remains is a loop that winds a number of times around $l$. This integer is by convention positive if the winding is counter-clockwise and negative otherwise. 
     \item   We now construct a linear chord diagram of class $a^w$ with $w$ being the winding number we just obtained. One can compute directly $w$ by the usual isomorphism $\pi_1 (D^2 \setminus \{l ,s\} ) \cong \Z$ and so $w \in \Z$ corresponds to the class of $\gamma (s) \in \pi_1 (D^2 \setminus \{l ,s\} )$. See Fig.~\ref{fig:LCD_CONSTR} bottom.
\end{itemize}

\begin{figure}[H]

 \centering

\includegraphics[height=9cm]{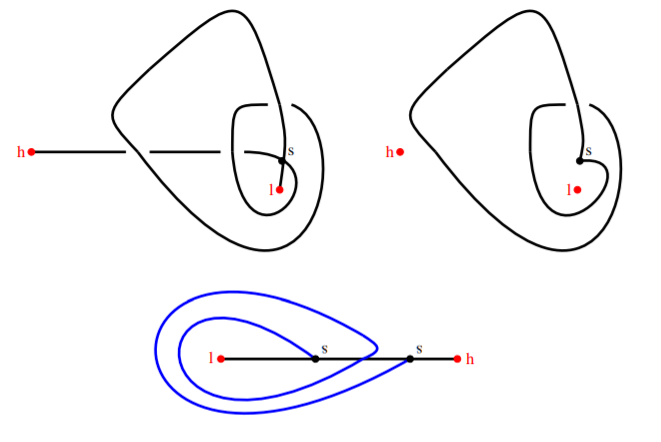}
  \caption{The singular loop and the corresponding chord diagram}
  \label{fig:LCD_CONSTR}
\end{figure}

The above generalize easily to an algorithm for corresponding a linear chord diagram to a singular spherical knotoid diagram $K$ of higher order:

\smallbreak


\begin{rem} \rm
Clearly, the above algorithm does not change under rigid vertex isotopy moves and under crossing switches. So, the linear chord diagram obtained by a singular knotoid diagram represents its whole singular equivalence class.  
\end{rem}

\subsection{From a linear chord diagram to a singular knotoid in \texorpdfstring{$S^2$}{.}}

It is a natural question to ask whether a linear chord diagram corresponds bijectively to a singular equivalence class. This question will only be answered in full (and in positive) in the next section for linear chord diagrams of order 1.

Here we will describe an algorithm for obtaining a singular spherical knotoid diagram from a linear chord diagram. 
 We will follow the classical construction of surgery along the two ends of a chord. 
The idea behind the technique is that a chord diagram illustrates the nodal structure of the singularities.  Furthermore, especially in pure knotoids the chords illustrate  the winding numbers of the singular loops around the leg.

\smallbreak

\noindent {\bf Algorithm for order 1 linear chord diagrams}

\begin{itemize}
    \item We start with a linear chord diagram with one chord, say $c$,  which has ends $a,b$ with $a$ being closer to $l$ and $b$ closer to $h$. See top left illustration in Fig.~\ref{fig:srg}.
    
    \item Take now a flat thickening of $c$: $c \times [-\e,\e]$  
    
    \item Cut out from the initial  interval the two small open sub-intervals $(a-\e,a+\e),(b-\e,b+\e)$. See top right illustration in Fig.~\ref{fig:srg}.
    
    \item The chord $c$  has now been duplicated into two new chords $c_-$ and $c_+$ comprising the remaining parts of the band's  boundary. See top right illustration in Fig.~\ref{fig:srg}.
    
    \item Introduce between the arcs $c_-$ and $c_+$  a singular crossing and at the same time a positive classical crossing, for retaining the connectivity pattern. See bottom left illustration in Fig.~\ref{fig:srg}. 
    
   \item  The result is an order one singular knotoid diagram, say $K$, which can be simplified by isotopy to the bottom right illustration in Fig.~\ref{fig:srg}. 
\end{itemize}

\begin{figure}[H]

 \centering
 
\includegraphics[height=9cm]{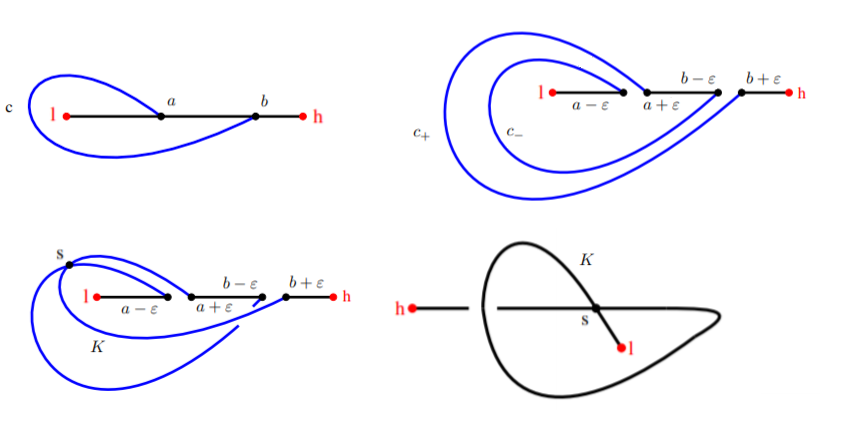}
  \caption{Recovering a singular knotoid from a chord diagram}
  \label{fig:srg}
\end{figure}

\begin{rem} \rm
Clearly, the resulting singular knotoid does not change under chord diagram equivalence (orientation preserving diffeomorphisms of the interval and isotopy of the immersed curve which fixes the endpoints).  Moreover, as it will become clear in the next section, the singular  knotoid obtained by a linear chord diagram is unique up to singular equivalence.
\end{rem}

\begin{rem} \rm
The above algorithms, the one for obtaining a linear chord diagram  out of a singular knotoid diagram  and the one for  obtaining a singular knotoid diagram out of a linear chord diagram, can be easily generalized for order greater than 1. Presenting them is beyond the scope of this paper, so they will be explained in a sequel work.
\end{rem}



\section{Type-1 invariants, regular diagrams and weight systems for spherical knotoids} \label{s:toi}

Thinking of classical finite type invariants, the one-term relation applies on trivial chord diagrams, yielding zero evaluation of any type-1 invariant, as illustrated in Fig.~\ref{fig:OTR}. In the theory of knotoids, however, as we already discussed,  linear chord diagrams are not trivial.

\begin{figure}[H]
    \centering
    \includegraphics[height=2cm]{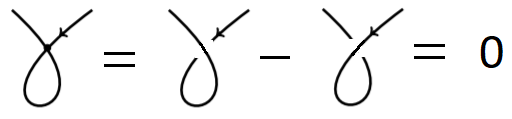}
    \caption{One-term relation}
    \label{fig:OTR}
\end{figure}
\subsection{The regular spherical knotoid diagrams for \texorpdfstring{$ v_1 $}{.}} \label{s:rkd}

Of crucial importance here is the notion of  regular diagram, which gives rise to our main result for any type-1 invariant, denoted generically $v_1$. 

\begin{defn}\rm\label{defn:descending}
A  knotoid diagram in a surface $\Sigma$ is \textit{descending} if when walking from leg to head every classical crossing encountered for the first time is an overcrossing. See also \cite{GUGUMCU2017186}. Similarly, a singular knotoid diagram in a surface $\Sigma$ is \textit{descending} if when walking from leg to head every segment traversed containing classical crossings is descending with respect to the classical crossings. 
\end{defn}

\begin{lemma}\label{lem:s_eq_desc}
Every singular knotoid diagram in a surface $\Sigma$ is singular equivalent to a descending singular knotoid diagram.
\end{lemma}

\begin{proof}
It it clear, since singular equivalence allows crossing switches. So,  applying appropriate crossing switches to the initial singular knotoid diagram, in can be deformed into a descending one.
\end{proof}

\begin{defn}\rm \label{defn:rkd}
A \textit{regular diagram} is a descending spherical knotoid diagram  with one singular crossing and with the minimal number of real crossings, among all diagrams in its rigid vertex isotopy class, and such that the singularity is located `next to'  the leg of the knotoid, in the sense that there exist no other crossings between the leg and the singular crossing. See Fig.~\ref{fig:rd}.
\end{defn}

Recall now Definition~\ref{defn:sequivalence} of singular equivalence. Our goal is to prove the following theorem.

\begin{thrm}\label{thrm:class}

Every spherical singular equivalence class of a knotoid diagram with one singularity contains exactly one regular diagram, and the regular diagrams are classified  by the winding number of the singular loop.

\end{thrm}

With this in hand we will be ready to show that we can handle the complexity occurring by the forbidden moves and the rigidity of the singularity just by the winding number of the singular loop.

Following Turaev \cite{turaev2012knotoids}, the spherical knotoid diagrams which have the head in the region of the point at infinity and the leg  in the region of  the point $0$ are called \textit{special}. Clearly we have the lemma:

\begin{lemma} \label{lem:special}

Every  spherical knotoid diagram $K$ is isotopic to a spherical knotoid diagram $\tilde{K}$, such that there is a path joining the head of $\tilde{K}$ with the point at infinity $ \infty $ and does not intersect any other part of the knotoid, that is,   $\tilde{K}$ is special.

\end{lemma}

\begin{proof}
Take a point $p\neq h$ in the connected component of $h$. By the fact that $S^2$ is rotationally symmetric, $p$ can be thought of as the point at infinity via an isometry.  Alternatively, one could also think that, given $h$ and $\infty$, we are not
allowed by the forbidden moves to pass branches of the knotoid $K$ through
$h$, but we can pass them through $\infty$ and, so, if in the initial position we
had $k$ intersection points when joining $h$, $\infty$, by passing $k$ branches to the other side of $\infty$ we have the desired special diagram only by using locally-planar isotopy. The same reasoning applies for $l$ and $0$.
\end{proof}

\begin{lemma} \label{lem:trsf}

The set of spherical knotoids $\mathcal{K}(S^2)$ is in bijection with the set of planar knotoids whose head lies in the non-bounded connected component of $\R^2$ minus the knotoid, which is viewed as  planar knotoid:

$$\mathcal{K}(S^2)\hookrightarrow K(\R^2) $$

\end{lemma}

We could equivalently think of these special planar knotoids as "long knots in one direction", in the sense that: even though the leg is trapped, yet following the knotoid diagram we can approach the head via a straight line outside a compact set. The same effect is achieved by identifying the head of a spherical knotoid with the point at infinity.

\begin{proof} \rm

Let $K\in \mathcal{K}(S^2)$. From Lemma~\ref{lem:special} we can assume, without loss of generality, that $K$ is special. By the decompactification of $S^2$, using the usual stereographic projection:
$$\pi : S^2 \setminus \{ \infty \}\longrightarrow \R^2 $$
the initial diagram $K$ corresponds uniquely to a diagram   $\tilde{K}\in K(\R^2)$.  The  diagram $K$ except a neighbourhood of $h$, is in a compact subset $B$ which does not contain $\infty$.  Then we have that $\pi(B)$ is compact in $\R^2$ and $\pi(h) \in \pi(S^2 \setminus B)$ which is non-bounded and has no common point with the knotoid other than the small path near $\pi(h)$. Moreover, we can assume that outside the compact subset $\pi(B)$ the knotoid $\tilde{K}$ reaches the head with a straight line since, from our hypothesis, $h$ is in the same region as $\infty$. The converse is analogous, since we begin with a knotoid $\tilde{K}\in K(\R^2)$, which lies inside a compact set $F$ except a simple straight arc from the boundary of $F$ to $h$. Applying $\pi ^{-1}$ yields a special knotoid which is what we want.

\end{proof}

\begin{figure}[H]
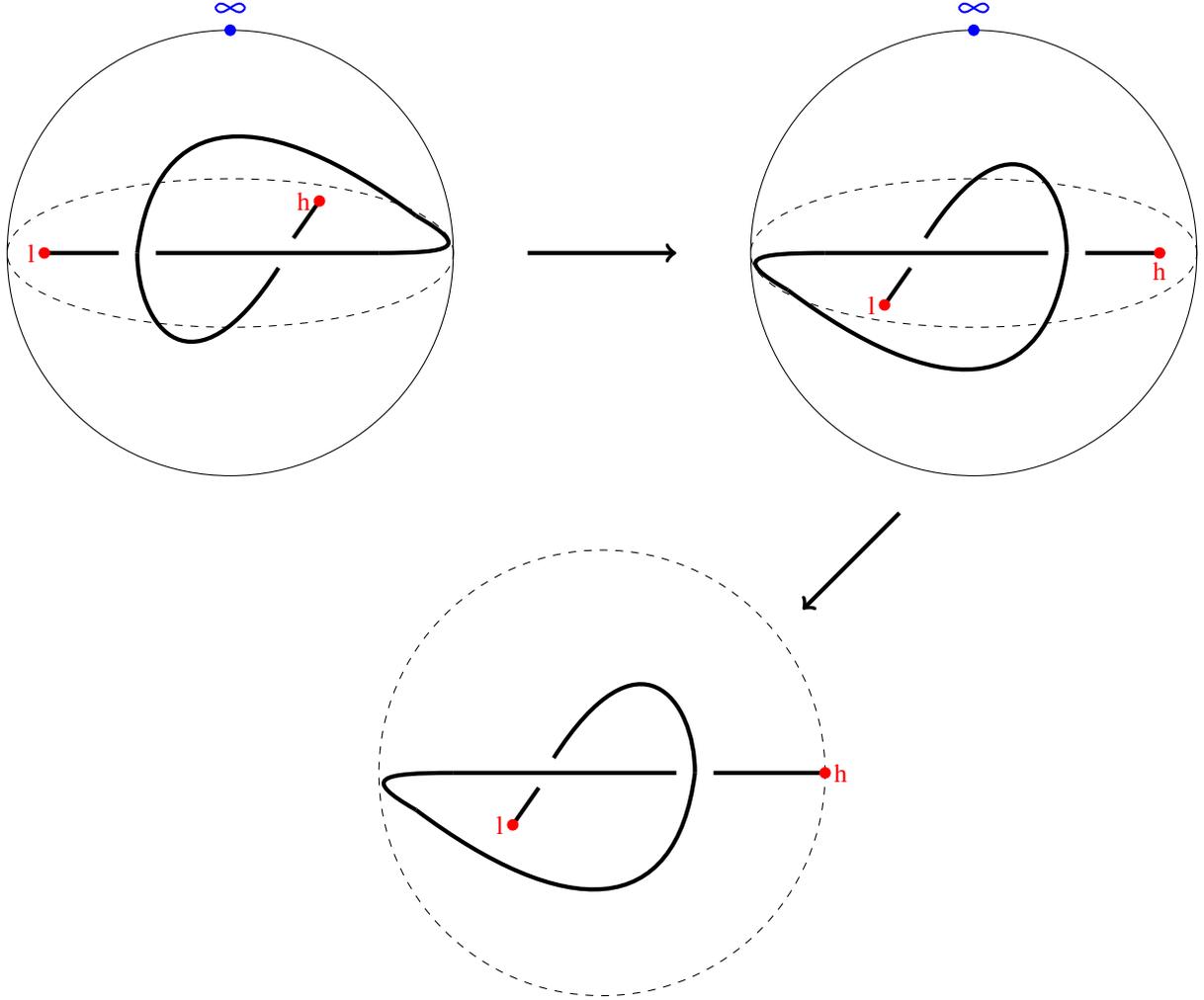

 
 \centering
 
 \trsf
  \caption{Transformation to a special knotoid diagram and then to a planar}
  \label{fig:trsf}
\end{figure}

\begin{lemma}\label{lem:desc}
Every descending  spherical knotoid diagram is isotopic to the trivial one.
\end{lemma}

\begin{proof} \rm
Since in the isotopy class of any spherical knotoids there is at least one special representative, we choose, without loss of generality, $K$ to be a descending special spherical knotoid diagram. 
  Take a disc $D$ in $S^2$ centered at $\infty$ which intersects $K$  only at a single straight segment $a$ starting from $h$. Then, the disc $S^2 \setminus D$ centered at $l$, after possible application of some planar isotopy. This disc encloses the rest of $K$ except for the segment $a$. 

Applying the stereographic projection $\pi$  yields in the plane the knotoid $\tilde{K}$ of Lemma~\ref{lem:trsf}, in which $h$ lies outside a compact region enclosing all the rest of $\tilde{K}$ except for the small straight segment $\pi (a)$ ending at $h$. 

We can now take the rail lifting of $\tilde{K}$ in $\R^3$. This gives rise to a solid cylinder with central axis the rail of $l$, denoted by $L$. See Fig.~\ref{fig:desc_triv}. Denote also by $H$ the rail of $h$. $H$ is allowed to move rigidly in the boundary of the solid cylinder.

\begin{figure}[H]
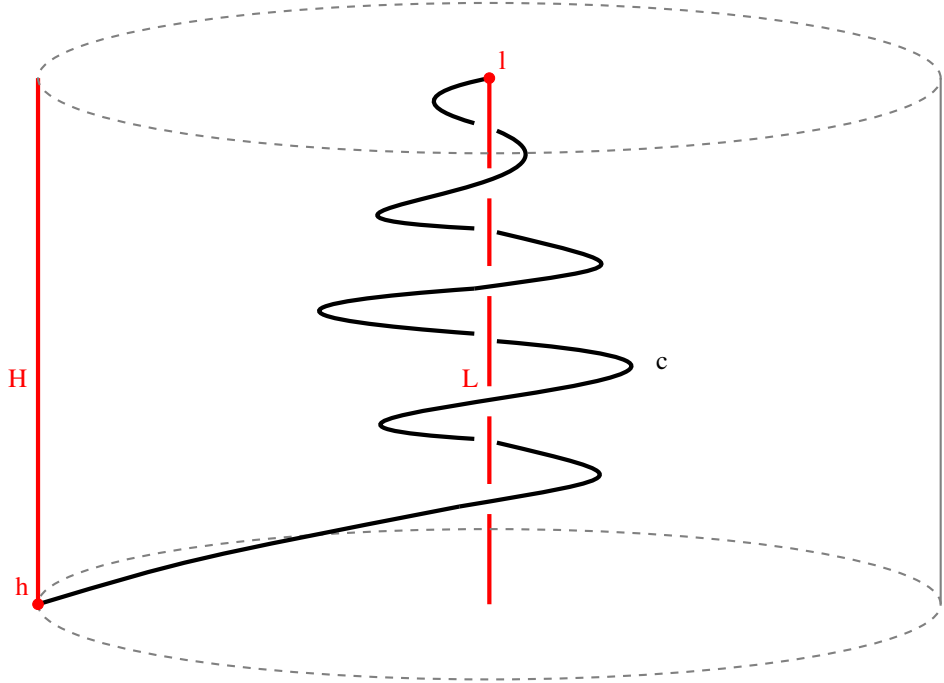

    \centering
    \RRKFL
    \caption{Rail lifting of a descending spherical knotoid }
    \label{fig:desc_triv}
\end{figure}

Now, our knotoid $\tilde{K}$ lifts to a curve $c$ in $\R^3$, parametrized by the unit interval, such that $c(0) = l$ and  $c(1) = h$.  The fact that  $\tilde{K}$ is descending means that if two points  $p_1=c(t_1), p_2=c(t_2)$ with $t_1<t_2$ and with $p_1=(x,y,z_1), p_2=(x,y,z_2)$, then $z_1>z_2$. Here we use the convention that $L,H$ are perpendicular to the $(x,y)$-plane, with $l$ fixed at the point $(0, 0, 1)$ on $L$  and $h$ lying in the $(x,y)$-plane at the lower end of the rail $H$. Hence, the curve $c$ is a helical curve that winds around the rail $L$ but does not wind around $H$, since $K$ is a special knotoid diagram. 

Now, the curve $c$ together with $L$, both oriented downward, form a $2$-braid. The parts that correspond to a braid word $\sigma_1\sigma_{1}^{-1}$ are cancelled by isotopy, so we have a braid word of the form $\sigma_1^n$, where $n\in \Z$. Now we can unwind $c$ from the bottom to the top by a continuous rotation of $H$ around $L$ in the direction opposite to the direction of the winding for $|n|$ half twists. Hence, the curve $c$ is line isotopic to a trivial arc and, so, the projection of $c$ on the $(x,y)$-plane is the trivial knotoid diagram.
\end{proof}

\begin{rem}\rm
Lemma~\ref{lem:desc} does not hold for planar knotoid diagrams. For example, the knotoid diagram of Fig.~\ref{fig:plan_desc} is not trivial.
\end{rem}

\begin{figure}[H]
    \centering
    \includegraphics[height=3cm]{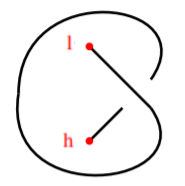}
    \caption{A non-trivial  descending planar knotoid diagram}
    \label{fig:plan_desc}
\end{figure}

\textbf{\textit{Proof of Theorem~\ref{thrm:class}}}

 Let $K$ be a singular spherical knotoid diagram with one singular crossing. By Lemma~\ref{lem:special}, $K$ can be considered to be special and, by Lemma~\ref{lem:s_eq_desc}, $K$ can be considered to be descending.
  Using, now, the injection of Lemma~\ref{lem:trsf}, we consider the rail lifting $c$ of $K$ in the solid cylinder  $C_R$  of radius $R$ centered at $L$. 
 
 Note that Lemma~\ref{lem:trsf} is valid also for spherical knotoid diagrams with one singular crossing where a neighbourhood of the knotoid diagram containing the singular crossing at the lifting can be considered to live in a flat disc, and in this point of view moves rigidly.
 
 The curve $c$ lives in the interior of $C_R \setminus L$, except $l$ and $h$, and the head $h$ can be considered to be situated in the line $H$, in the boundary of $C_R$.  $c$ comprises three segments: ($l \rightarrow s$) from $l$ to $s$,  the singular loop ($s \rightarrow s$), and ($s \rightarrow h$) from $s$ to $h$. 

Since $K$ is descending,  up to space isotopy, the segment ($l \rightarrow s$) can be viewed as a classical descending spherical  knotoid and, hence, by Lemma~\ref{lem:desc}, it is an unknotted arc that lies above the rest of $c$. This arc can be straightened by isotopy, employing the unwinding technique that we described earlier, that is, by rotating around the rail $L$. Using isotopy, one may ensure  that after the first encounter of the point $s$, the knotoid lies outside the cylinder centered at $L$ with radius $d(s,L)$. See Fig.~\ref{fig:rrkf}. 
Further on, again by the fact that $K$ is descending and using space isotopy, the singular loop ($s \rightarrow s$) can be viewed to wind around $L$ with decreasing radii. See Fig.~\ref{fig:rrkf}. Finally,  by the same arguments and employing the unwinding technique, the segment ($s \rightarrow h$) is also an unknotted  straight arc, descending from $s$ to $h$. See Fig.~\ref{fig:rrkf}.

\begin{figure}[H]
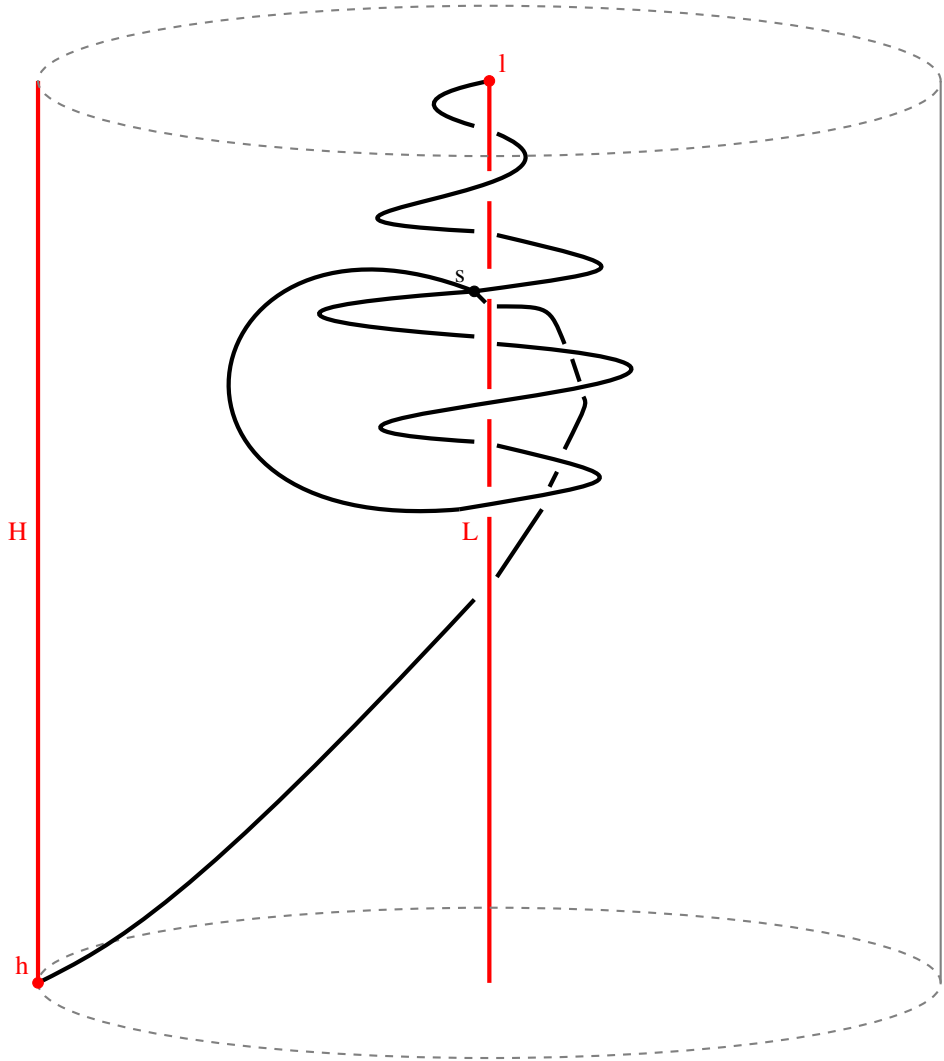

 \centering
 \RRKF
  \caption{Rail lifting of a spherical singular descending knotoid }
  \label{fig:rrkf}
\end{figure}

We now pass to the projection of the thickened annulus $C_R \setminus L$ in the $(x,y)$-plane.  $C_R \setminus L$ projects to the annulus $D^2 \setminus\{l\}$. Then the curve $c$ projects to a singular knotoid $\bar{K}$ comprising three segments analogous to the ones of $c$ and for which we shall use the same notation: ($l \rightarrow s$) from $l$ to $s$,  the singular loop ($s \rightarrow s$), and ($s \rightarrow h$) from $s$ to $h$. From the above, the subarc ($l \rightarrow s$) contains no crossings. The singular loop ($s \rightarrow s$) is descending and contains $w-1$  crossings, where $w$ is its winding number around $l$. Finally, the segment ($s \rightarrow h$) lies entirely under the rest of the diagram and creates $w$ crossings with the singular loop. 

Therefore, by Definition~\ref{defn:rkd} $\bar{K}$ is a regular diagram. Recall, for example, the top left diagram in Fig.~\ref{fig:LCD_CONSTR}. Moreover this minimal number of crossing is always odd as we see that such a diagram has always $2w-1$ real crossings.

 We finally observe that, by definition, two different regular diagrams correspond to two distinct singular loops. Now, a singular loop can be viewed as an element of  the fundamental group of the annulus $\pi_1 (D^2 \setminus\{l\}, s)$, which is $\Z$. Hence, the singular loops  and, consequently, the regular diagrams are classified by $\Z$. 

From the above, each singular equivalence class has exactly one distinct regular diagram as a representative. \qed

\begin{cor} \label{cor:bij}
Combining the results of Section~\ref{s:lcd} and Theorem~\ref{thrm:class} we have that the set $A_1^l$ of all spherical linear chord diagrams of order 1  corresponds  bijectively to the set of regular knotoid diagrams.
\end{cor}

\subsection{The integration theorem}

With the proof of Theorem~\ref{thrm:class} and by Corollary~\ref{cor:bij}, a natural question arises.  Namely,  given a {\it weight system} $W$, which is a function $W: A_1^l\rightarrow R$ where $R$ is a commutative ring,  does $W$ determine a type-1 invariant of spherical knotoids? The first, obvious thing to observe is that $W$ must give the value zero to the trivial linear chord diagram, since this diagram corresponds to the unknot with one singular crossing. 

We start exploring the question by the following idea. Given a spherical knotoid diagram $K$ we bring $K$ to a descending form, by appropriate crossing switches and creating at the same time a number of singular knotoid diagrams with one singular crossing. More precisely, we apply the following algorithm.
\begin{itemize}
    \item Start walking from the leg of $K$ toward its head. 

  \item  Let $c_1$ be the first crossing that we traverse from underneath as we arrive at it for the first time. We switch $c_1$ and we obtain a new knotoid diagram $K_1$  and  a singular knotoid diagram $S_1$. 
  
  \item  We repeat the above step for the diagram $K_1$  and its first crossing  $c_2$ that fails $K_1$ to be descending, obtaining a new knotoid diagram $K_2$  and  a singular knotoid diagram $S_2$. 
    
  \item After an appropriate number of repetitions, say $r$, we arrive at a descending diagram $K_r$, which is trivial, and a singular knotoid diagram $S_r$, and so the algorithm terminates.
  
  \item Next, correspond each $S_i$ to the linear chord diagram $L_i$ representing the regular diagram in its singular equivalence class.
  
 \item Then evaluate each $L_i$ according to the weight system $W$. 
 
 \item Finally, define  
 \begin{equation}\label{eqn:weight}
     v_1(K) := \sum_{c\in Cr(K)} \de_c sgn(c)W(K_c)
 \end{equation}
where $$\de_c=  \begin{cases} 0, & \mbox{c is an over crossing the first time} \\1, & \mbox{otherwise} \end{cases}$$ 
$sgn(c)$ is the sign of the crossing $c$ and $W(K_c)$ is the value of $W$ on the linear chord diagram that corresponds to the singular knotoid diagram $K_c$, which is $K$ with $c$ nodified.
\end{itemize}

\begin{thrm}\label{thrm:integr}

Every weight system $W:A^l_1\rightarrow R$ that has zero evaluation on the trivial linear chord diagram, gives rise to a type-1 spherical knotoid invariant, by means of  formula (\ref{eqn:weight}). Equivalently,  the sum (\ref{eqn:weight}) is constant on the isotopy class of any spherical  knotoid diagram.
\end{thrm}

\begin{proof} 
We shall prove that  the sum (\ref{eqn:weight}) remains invariant under the application of Reidemeister moves. 

Indeed, consider first a Reidemeister~I move, creating a crossing, say $c$. Then the diagram  $K_c$ with the crossing nodified corresponds to the trivial linear diagram, on which the weight system is $0$. So Reidemeister~I moves contribute nothing to our sum.

A Reidemeister~III move does not create any crossings and it does not hide any of the crossings involved. Furthermore, in an oriented Reidemeister~III move there are some choices involved, because of the orientations. In any such choice, we will see the contribution of the change to be zero. Consider the case where all three arcs  illustrated in Fig.~\ref{fig:RTHREEB} are oriented from left to right. The central crossing contributes to the sum the same before and after the  move, since the two corresponding singular diagrams are rigid vertex isotopic (recall Fig.~\ref{fig:skmoves}). Let now $A$ be the left crossing and $B$ the right one before the move is performed, and let $A^\prime, B^\prime$  be their corresponding crossings after the  move. See Fig.~\ref{fig:RTHREEB}. $O$ is the initial point, i.e. the first point of our local picture that we pass through, travelling from the leg to the head of the knotoid diagram.

\begin{figure}[H]
    \centering
    \includegraphics[height=3cm]{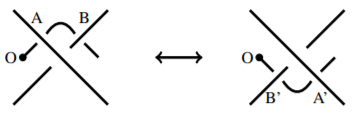}
    \caption{Study of the summation contribution of a Reidemeister~III move}
    \label{fig:RTHREEB}
\end{figure}

Then $K_A$ and $K_{A^\prime}$ correspond to the same regular diagram, since neither the leg nor the head are in the local region of the move, or else we performed a forbidden move, and so the winding number of the singular loop does not change. Hence, $W(K_A) = W(K_{A^\prime})$ and $W(K_B)=W(K_{B^\prime})$. Moreover, $\de_A = \de_{A^\prime},\de_B = \de_{B^\prime}$ since the same arcs are involved and of course $sgn(A) = sgn(A^\prime), sgn(B) = sgn(B^\prime)$. Therefore, again, there is no contribution to our sum. All other cases of Reidemeister~III moves are treated analogously.

Last we check the Reidemeister~II move. The idea is again to perform such a move and show that it contributes nothing to the sum. In Fig.~\ref{fig:RTWOB} we sum over diagrams, for simplicity, meaning the evaluations on these diagrams. In the diagrams in Fig.~\ref{fig:RTWOB}  we choose the orientations from top to bottom and we assume that the understrand is approached first from a local initial point $O$. 

\begin{figure}[H]
    \centering
    \includegraphics[height=8cm]{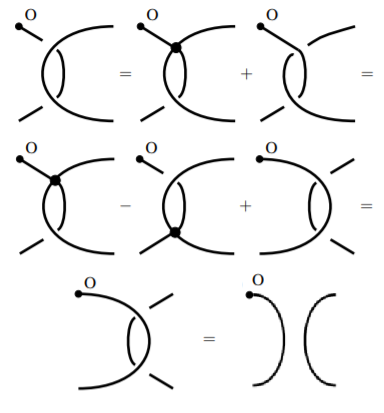}
    \caption{Study of the summation contribution of a Reidemeister~II move}
    \label{fig:RTWOB}
\end{figure}

Then, invoking twice the Vassiliev skein relation, we end up with the diagram where both crossings in the move are first overcrossings. See Fig.~\ref{fig:RTWOB}. But this is invisible by the algorithm and, consequently, by the summation. So, the Reidemister~II move contributes  to the sum 
nothing more than the diagram with two parallel lines in the local region of the move. The proof for all other choices is analogous.
\end{proof}

\begin{rem}\label{rem:v_1_knot_trivial}\rm

Every knot invariant of type-$1$ is trivial. 

\end{rem}
Indeed, let $v$ be an invariant of type one and let $K$ be a singular knot  with exactly one singular crossing  ($v$ vanishes on knots with more singularities). The singular crossing divides the knot into two disjoint closed curves. After appropriate classical crossing switches, using the Vassiliev skein relation,  the two closed curves can turn into simple unknotted closed curves, with no self-crossings or shared crossings between them. Now, these crossing switches cost nothing, because each use of Eq.~\ref{eqn:vskein} will add a singular crossing to the one of the two resulting diagrams and, by definition, $v$ will vanish on this diagram. Hence, we derive the relation: 
$$
v(K := \SC) = v(\finfty)
$$ 
which, again, using Eq.~\ref{eqn:vskein}, is clearly equal to zero, by the first Reidemeister move.

\begin{cor}
Theorems~\ref{thrm:class} and~\ref{thrm:integr} imply that there are non-trivial type-1 invariants for spherical knotoids. Therefore, there are also  non-trivial type-1 invariants  also for planar knotoids since, by  Lemma~\ref{lem:trsf}, the set of spherical knotoids corresponds bijectively to the set of planar knotoids whose head is in the non-bounded component of $\R^2$ minus the knotoid. The precise theory of classifying type-1 invariants  for planar knotoids is beyond the scope of the present paper  and is the subject of future work. 
\end{cor}

\section{The singular height} \label{s:sh}

In this section we extend the notion of `height'  (or complexity) of classical knotoids  to singular knotoids, and we introduce the notion of `singular height'for singular knotoids. Both, the height and the singular height are invariants of singular knotoids.  We conclude the section with a conjecture on the singular height of planar singular knotoids with one singularity.

\begin{defn}\rm
The \textit{height}, $h(K)$, of a singular knotoid diagram $K$ is the minimal number of transversal double point intersections of all possible shortcuts connecting the two endpoints of $K$ with the  diagram $K$.  Furthermore, the \textit{height}, $h(k)$,  of a singular knotoid $k$ is defined to be the minimal height over all possible diagrams in the rigid vertex isotopy class of $k$.
\end{defn}

We further define the singular height of a singular knotoid.

\begin{defn} \rm\label{defn:sh}

The \textit{singular height}, $sh(k)$, of a singular knotoid $k$ is the minimal height over all singular knotoid diagrams in the  singular equivalence class of $k$. Obviously $sh(k)\le h(k)$.
\end{defn}

We showed in Theorem \ref{thrm:class} that the singular equivalence class of a spherical singular knotoid with one singularity can be represented by a distinct regular diagram.  We then have:

\begin{lemma}\label{lem:height_regular}
The height of a regular diagram $R_n$ whose singular loop winds $n$ times around the leg equals $n$, that is, $h(R_n)=n$.
\end{lemma}

\begin{proof}
Let $R_n$ be the regular knotoid diagram whose singular loop winds $n$ times around the leg of the singular knotoid. We first note that the winding number of the singular loop around the leg is a homotopy invariant in the disc $D^2$ punctured by the point $l$, thus it cannot be reduced. By $n$ consecutive applications of the Jordan curve theorem, one needs to intersect the singular loop at least $n$ times in order to join the leg with the head with a simple arc. On the other hand, there is obviously a shortcut for $R_n$  that creates precisely $n$ intersections. So $h(R_n)=n$.
\end{proof}

\begin{rem}\label{rem:interwind} \rm 
The winding number of the singular loop is equal to the algebraic intersection number of an arc joining the singular point $s$ with the head of the knotoid diagram.
\end{rem}

\begin{prop}
 A regular diagram realizes the singular height of its singular equivalence class.$sh(k_n)=h(R_n)$
\end{prop}

\begin{proof}
By Theorem~\ref{thrm:class} a singular equivalence class is  characterized by  the winding number of the singular loop around the leg, which is a  homotopy invariant in the annulus, thus it cannot be reduced.  Let now $k_n$ denote the singular equivalence class with singular loop with winding number $n$. Applying  the idea of the proof of Lemma~\ref{lem:height_regular} on any diagram in the class $k_n$, means from the above that this diagram cannot have height less than $n$. We further have that the regular diagram $R_n$ belongs to $k_n$ and, by Lemma~\ref{lem:height_regular}, $h(R_n) = n$. Therefore, $sh(k_n)=h(R_n) = n$.
\end{proof}

Thus, it is natural to say that since a regular diagram realizes height and a linear chord diagram corresponds to exactly one regular diagram, chord diagrams realize singular height.

The question here is simple: In planar knotoids can we detect singular height, just by seeing the word corresponding to the chord diagram?

The farther we can reach for the moment is to conjecture the formula that holds in this occasion and search it further in the future.
First of all recall that in planar knotoids $b$ is a free generator as well as $a$. We don't have that $ab=1$ as in Lemma~\ref{lem:a*a} so we should think linear chord diagrams for planar knotoids as words of $F_2[a,b]$. 

So let $c\in F_2[a,b]$, then $c=a^{n_1}b^{m_1}\dots a^{n_k}b^{m_k}, k<\infty,n_i,m_i\in\Z$
\begin{itemize}
    \item $m_i \neq 0 \quad \forall i\le k-1$
    \item $n_i \neq 0 \quad \forall i\ge 2$
    \item if $m_1=0 \Rightarrow m_i=0$ for every $i$ and then $c=a^r$ for some $r\in \Z$
\end{itemize}

We search the first two consecutive exponents which have the same sign.

\begin{enumerate}
    \item If the first such pair is $n_j>0, m_j>0$ for some $j$, then let $d$ be the number of (consecutive) pairs $n_i,m_i$ where $n_i>0, m_i>0$ added by the number of (consecutive) pairs $m_i,n_{i+1}$ where $m_i<0, m_{i+1}<0$
    \item If it is $m_j<0, m_{j+1}<0$ let $d$ be the same as in $(1)$.
    \item If the first such pair is $n_j<0, m_j<0$ for some $j$, then let $d$ be the number of (consecutive) pairs $n_i,m_i$ where $n_i<0, m_i<0$ added by the number of (consecutive) pairs $m_i,n_{i+1}$ where $m_i>0, m_{i+1}>0$
    \item If it is $m_j>0, m_{j+1}>0$ let $d$ be the same as in $(3)$
\end{enumerate}

\begin{conj}
Given a linear chord diagram $D$ with one chord, for planar knotoids, then the singular height of the corresponding singular planar knotoid diagram is given by the formula
$$
\displaystyle{sh(k)=\left[\sum_{i=1}^k\left(|n_i|+|m_i|\right)\right]-2d}
$$
\end{conj}

\section{Examples of type-1 invariants for  knotoids and a universal invariant} \label{s:exs}

Any invariant $I$ of classical (or virtual) knots can be used to define an invariant of knotoids by defining the new invariant of knotoids to be $I$ applied to the classical (resp. virtual)  closure of the knotoid.

\subsection{The affine index polynomial}\label{sub:affine}

The affine index polynomial was defined for virtual knots
and links by L.H. Kauffman \cite{doi:10.1142/S0218216513400075, FolwacznyKauffman}, generalizing and reformulating an invariant of A. Henrich in her PhD thesis (see~\cite{Henrich}), and then for knotoids, virtual or classical, by N. Gügümcü and L.H. Kauffman \cite{GUGUMCU2017186}. It is based on an integer labeling assigned to flat knotoid diagrams (i.e. diagrams with the information `under' or `over' on classical crossings omitted) in the following way. A flat knotoid diagram, classical or virtual, is associated with a graph where the flat classical crossings and the endpoints are regarded as the vertices of the graph. An arc of an oriented flat knotoid diagram is an edge of the graph it represents, that extends from one vertex to the next vertex. 

We start labeling the edges as exemplified in Fig.~\ref{fig:label}. At each flat crossing, the labels of the arcs change by one. If the incoming arc labeled by $a\in\Z$ crosses another arc which goes to the right then the next arc is labeled by $a+1$;  if the incoming arc $b\in\Z$ crosses another arc which goes to the left then it is labeled by $b-1$. There is no change of labels at virtual crossings. Note that it is convenient to label the first arc from the leg to the first crossing by $0$.

\begin{figure}[H]
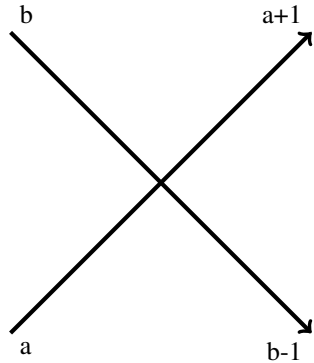

    \centering
     \KP{
     \draw[color=black, ultra thick,->] (-2,-2) -- (2,2);
      \draw[color=black, ultra thick,->] (-2,2) -- (2,-2);

     \filldraw [black] (-2,-2) circle (0.0001pt)  node[anchor=north west] {a};
     \filldraw [black] (2,2) circle (0.00001pt) node[anchor= south east] {a+1};
     \filldraw [black] (-2,2) circle (0.000001pt) node[anchor=south west] {b};   
       \filldraw [black] (2,-2) circle (0.00000001pt) node[anchor=north east] {b-1};  
  }%
    \caption{Integer labeling}
    \label{fig:label}
\end{figure}
Let $c$ be a classical crossing of a knotoid diagram $K$ in a surface $\Sigma$. We denote $w_+(c), \, w_-(c)$ the following integers derived from the labels at the  flat crossing corresponding to $c$:
$$
w_+(c) :=b-(a+1)
$$
$$
w_-(c) :=a-(b-1)
$$
where $a$ and $b$ are the labels for the left and the right incoming arcs,  respectively. The numbers $w_+(c)$ and $w_-(c)$ are called positive and negative weights of $c$, respectively. We also denote $sgn(c)$ the sign of the crossing $c$.
Define, now, the {\it weight} of $c$ to be: 
$$
w_K(c) :=w_{sgn(c)}(c)
$$

\begin{defn} \rm\label{defn:aip}
The {\it affine index polynomial} of knotoid diagram $K$ in a surface $\Sigma$  is defined as
$$
\displaystyle{P_K(t) := \sum_{c\in Cr(K)}sgn(c)(t^{w_K(c)}-1)}
$$
where $Cr(K)$ is the set of all classical crossings of $K$.
\end{defn}

The affine index polynomial is an isotopy invariant and  has some very interesting properties, as shown in \cite{GUGUMCU2017186}, including that its higher degree is smaller than or equal to the height of the knotoid. 
Clearly, for a knot-type knotoid, or equivalently for a knot, the affine index polynomial is trivial.
Furthermore, it is very easy to calculate it by hand, as in the following example which will be very useful in what follows.

\begin{exa} \rm \label{ex:aff}
 We will calculate the affine index polynomial of the knotoid below:

\begin{figure}[H]
    \centering
    \includegraphics[height=7cm]{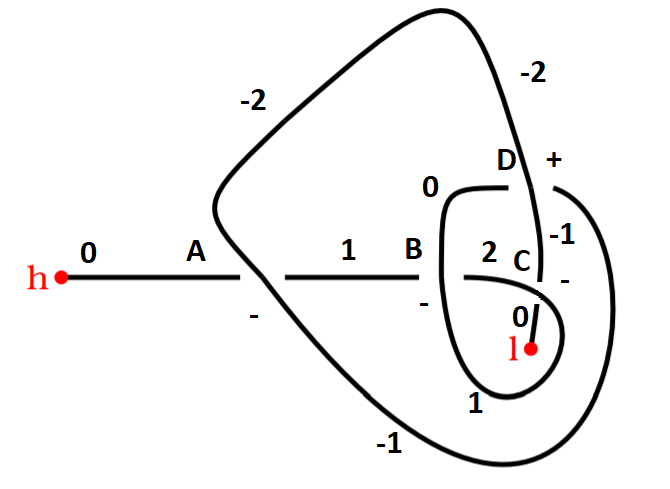}
    \caption{Labeling the knotoid}
    \label{fig:ex_aff}
\end{figure}

\begin{table}[H]
    \centering
    \begin{tabular}{c| c| c|}
         & $w_+=b-(a+1)$  & $w_-=a-(b-1)$\\
         \hline
        A & $2$ & $-2$\\
        \hline
        B & $1$ & $-1$\\
        \hline
        C & $-2$ & $2$\\
        \hline
        D & $-1$ & $1$\\
        \hline
    \end{tabular}
    \caption{Crossing weights for the affine index polynomial}
    \label{tab:weights_affine}
\end{table}
\noindent With the labeling of Fig.~\ref{fig:ex_aff} and of Table~\ref{tab:weights_affine} we have everything we need for calculating the polynomial.  Namely:

$$
\displaystyle{P_K=\sum_{c\in \{A,B,C,D\}}sgn(c)(t^{\omega_k(c)}-1)=-(t^{-2}-1)-(t^{-1}-1)-(t^2-1)+(t^{-1}-1)=}
$$
$$
\displaystyle{=-(t^2+t^{-2}-2).}
$$
Note that, by the discussion before the example, we have a simple proof that $K$, as illustrated in Fig.~\ref{fig:ex_aff}, has height $2$. 
\end{exa}

\begin{prop} 
The affine index polynomial is a type-1 invariant for knotoids in a surface $\Sigma$. 
\end{prop}

\begin{proof}
We extend the affine index polynomial to singular knotoid diagrams using the Vassiliev skein relation:
$$
P\left(\SC\right) = P\left(\PC\right) - P\left(\NC\right) 
$$
The two diagrams in the right-hand side of the equality differ only in the place of the singular crossing. Applying on each one the formula $ 
\displaystyle{P_K(t)=\sum_{c\in Cr(K)}sgn(c)(t^{\omega_K(c)}-1)}
$  of  Definition~\ref{defn:aip}, we have that in the difference   of the affine index polynomials the contributions of any other crossing cancel each other. 
 So, for a singularity inserted in a knot diagram we have the following contribution to the sum 
$$ 
P\left(\SC\right)  = (t^{\omega_+(c)}-1)-(-(t^{\omega_-(c)}-1))=t^{\omega_+}+t^{\omega_-}-2. 
$$ 
The constant term $-2$ is justified by the change of writhe by $2$ when one changes a positive crossing to a negative.

Take, now, a knotoid diagram with two singularities $c, d$ with corresponding weights $\omega_{\pm}$ and $\omega^{\prime}_{\pm}$. 
From the above we have the following contributions: 

$$
\displaystyle{P\left[ \SC_c, \SC_d\right]}=P_{++}-P_{+-}-P_{-+}+P_{--}=$$
$$
(t^{\omega_+}-1)+ (t^{\omega^{\prime}_+}-1)-(-t^{\omega_-}+1)-(t^{\omega^{\prime}_+}-1)-(t^{\omega_+}-1)-(-t^{\omega^{\prime}_-}+1)+(-t^{\omega_-}+1)+(-t^{\omega^{\prime}_-}+1)=0
$$
So $P_K$ is a Vassiliev invariant of knotoids of type $1$.
\end{proof}

Now recall the  Example~\ref{ex:aff} where we computed the affine index polynomial of the knotoid $K$ of Fig.~\ref{fig:aff_calc_reg}. It was $P_K=-(t^2+t^{-2}-2)$. We also have (omitting the $P_K(t)$) the computation illustrated in Fig.~\ref{fig:aff_calc_reg}.

\begin{figure}[H]
    \centering
    \includegraphics[height=4cm]{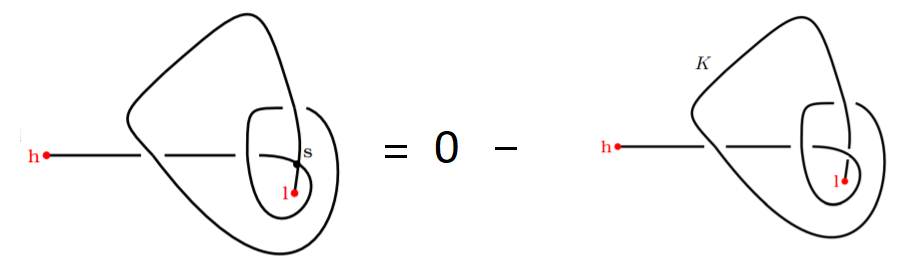}
   \caption{ The affine index polynomial of the knotoid $K$ is  $P_K = -(t^2+t^{-2}-2)$}
    \label{fig:aff_calc_reg}
\end{figure}

So the regular diagram with winding number $2$ has an affine index polynomial $P_K=t^2+t^{-2}-2$, which of course was what we expected, by the calculation of the affine index polynomial on an abstract knotoid diagram with one singularity.

\begin{rem}\label{rem:affwind} \rm

By \cite[proof of Theorem 4.12]{GUGUMCU2017186}, an equivalent way to calculate the weight of a crossing is by taking the sum, say $n$, of the algebraic intersection numbers of the loop $l(c)$ formed by the crossing $c$. Then $n$ is equal to either $w_-(c)$ or $w_+(c)$, depending on the orientation of $l(c)$.
Applying the above for $c$ the crossing replacing the singular crossing $s$ of a spherical knotoid, we have that $l(c)$ corresponds to the singular loop. Then, from Remark~\ref{rem:interwind}, the winding number,  $w$,  of the singular loop is equal to either $w_+(c)$  or $w_-(c)$. 
Therefore, the affine index polynomial is a type-1 invariant for spherical knotoids that detects the winding number of the singular loop around the leg. Yet, it does not detect the direction, as every singular knotoid has its affine index polynomial evaluation equal to the expression $t^w+t^{-w}-2$ which has symmetric exponents. 
   Hence, an analogue of the affine index polynomial might be able to   detect every singular equivalence class. We want to find a special weight assignment which gives rise to a type-1 invariant that does the trick. This is done in the next subsection.
\end{rem}

\subsection{The universal  \texorpdfstring{$ \bar{v} $}{} invariant}

As shown by the integration theorem (Theorem~\ref{thrm:integr}), any weight system $W$ gives rise to a type-1 knotoid invariant iff it respects the one-term relation, via the formula $$v_1(K)=\sum_{c\in Cr(K)} \de_c sgn(c)W(K_c)$$

As pointed by the Remark~\ref{rem:affwind} the $t^w+t^{-w}-2$ of the affine index polynomial was close to have the desired properties.

Let $K_s$ be a regular diagram (with one singular crossing), such that the singular loop winds around the leg $w$ times. Choose:
$$W(K_s) := t^w-1.$$
The trivial linear chord diagram, $t$, does not wind around the leg, so  $w_t =0$ and  $W(K_t) = t^0-1 =0$. Hence this weight system respects the one-term relation.
\begin{defn} \rm 
We define:
$$\bar{v}(K) :=\sum_{c\in Cr(K)} \de_c sgn(c)W(K_s)=\sum_{c\in Cr(K)} \de_c sgn(c)(t^{w_c}-1)$$
As it follows by the integration theorem, $\bar{v}$ is a type-1 knotoid invariant.
\end{defn}

\begin{cor}
Since every knotoid with one singularity has a $\bar{v}$ evaluation equal to $t^w-1$, $\bar{v}$ is a winding number detector.
\end{cor}

In the same way we can define a new invariant $\zeta$ by choosing $W(K_s) :=t^w+t^{-w}-2$. The proof that $\zeta$ is an invariant under isotopy is identical as the one for the invariant $\bar{v}$.

\begin{prop}
The invariant $\zeta$  defined by the weight system $W(K_s) :=t^w+t^{-w}-2$ is the affine index polynomial.
\end{prop}

\begin{proof}
Both the $\zeta$-invariant and the affine index polynomial are type-1  invariants for knotoids. Moreover, as we already saw, in each regular diagram they take values that depend only on the winding number, namely $\zeta(R_n)=P(R_n)=t^n+t^{-n}-2$, where $R_n$ is the regular diagram with winding number of its singular loop around the leg equal to $n$.

Since their top row evaluations are identical we get that they differ only by type-0 invariants and, so, they differ by their trivial knotoid evaluations, which vanish in both cases.
Hence, the two invariants are identical. 
\end{proof}

\begin{rem}\rm
The affine index polynomial and the $\bar{v}$ invariant are related through the following equation: 
$$P_K(t)=\bar{v}(K)(t)+\bar{v}(K)(t^{-1}).$$
This is a consequence of the conversation above, just by comparing the weight assignments.
\end{rem}

\begin{rem}\rm \label{rem:Henrich}
A. Henrich showed that an analogue of the affine index polynomial is a type-1 invariant for virtual knots. Furthermore, Henrich \cite{Henrich} has a general theory of type-1 invariants for virtual knots and a direct approach for a universal type-1 invariant for virtuals, which is called the `Gluing invariant' and is denoted by $G$. In this context, one can adapt the definition of the Gluing invariant for knotoids by  defining:  $G^{\prime}(k) = G(\bar{k})$, where $\bar{k}$ denotes here the virtual closure of the knotoid $k$. See also related work of Petit~\cite{Petit}. 

As pointed out in Remark~\ref{rem:clos_inv} taking invariants of knotoids via any type of closure loses topological information of the knotoids. If, now, we pull back the Gluing invariant and define it directly on knotoid diagrams, it would be in our notation:  
$$
\displaystyle{G(K) := \sum_{c\in cr(K)}sgn(c)(K_c - K_t)}
$$
where $K_t$ is the equivalence class of the trivial knotoid diagram with one singularity. 
 We note that  $K_t$ is unique  because all spherical knotoid diagrams are equivalent when seen up to homotopy, while on the other hand there are many different homotopy classes of virtual knots.  The Gluing invariant can, thus, also be regarded as a universal  type-1 invariant for knotoids. We can do something more. Our direct approach is a specific weight system driven formulation of the type-1 invariant, which produces our universal $\bar{v}$  invariant above. In fact, our formulation can be generalized to Vassiliev invariants of higher type for knotoids and it relates every invariant to a weight system of chord diagrams.

 \end{rem}

\section{Higher type invariants coming from the Jones polynomial, the Kauffman bracket and the Turaev extended bracket} \label{s:hoti}

\subsection{The Kauffman bracket and the Turaev extended bracket for knotoids}

 In \cite{turaev2012knotoids} several invariants for knotoids have been defined, mainly  for spherical ones. One of them is the  {\it Kauffman bracket}  polynomial. Namely, for $K$ a spherical knotoid or  multi-knotoid diagram we define inductively  $\left\langle K\right\rangle$ to be the element of the ring $\mathbb{Z}[A^{\pm 1}] $   by means of the rules:

\begin{enumerate}
\item
  $\left\langle\KPA\right\rangle = 1 = \left\langle\KPF\right\rangle$ 

\item 
  $\left\langle L \cup \KPA\right\rangle=(-A^{2}-A^{-2})\langle L\rangle$

\item
  $\left\langle\KPB\right\rangle=
  A\left\langle\KPC\right\rangle + A^{-1} \left\langle \KPD \right\rangle$
  
  \end{enumerate}
  
Proceeding now with a closed formula:  If $G$ is the underlying planar graph for $K$, then a state of $G$ is a choice of splitting marker
for every vertex of $U$.  We can choose between the $A$-splitting which is $\KPC$ or the  $B$-splitting which is $\KPD$. We call the underlying planar graph for a diagram $K$ the {\it universe}  for $K$. This terminology distinguishes the underlying planar graph from the link projection and from other graphs that can arise. Thus we speak of the states of a universe.
Then the formula for the bracket is $$\displaystyle{\left\langle K\right\rangle=\sum_{S:state}\left\langle K|S\right\rangle (-A^2-A^{-2})^{|S|-1}=\sum_{S:state} A^{\sigma_s} (-A^2-A^{-2})^{|S|-1}}$$ 
where $\sigma_s\in \mathbb{Z} $ is the sum of the values $\pm1$ of the states over all crossings of $K$.

The bracket polynomial is a Laurent polynomial ( $\left\langle K\right\rangle \in \mathbb{Z}[A^{\pm 1}] $ ), and it is a regular isotopy invariant. 

Under a  Reidemeister~I move the bracket polynomial is multiplied by $-A^{\pm 3}$, so $(-A^3)^{-w(K)} \left\langle K\right\rangle$ is an isotopy invariant, where $w(K)$ is the writhe of $K$, and (for classical knots) it coincides with the {\it Jones polynomial} with an appropriate change of variable.

We shall now recall the definition of the {\it Turaev extended bracket}, which is a two-variable extension of the Kauffman bracket.
\begin{itemize}
    \item Pick a shortcut $\alpha \in S^2$ for a spherical knotoid diagram $K$.
    \item given a state $s$ call the smoothed 1-manifold $K_s$ and its segment component $k_s$.  $k_s$ coincides with $K$ in a small neighborhood of the endpoints of $K$, and  $\partial k_s = \partial \alpha$ consists of the endpoints of $K$.
    \item orient $K$, $k_s$, and $\alpha$ from the leg to the head of $K$.
    \item Set $k_s\cdot \alpha$ be the algebraic number of intersections of $k_s$ with $\alpha$ and $K\cdot \alpha$ be the algebraic number of intersections of $K$ with $\alpha$, with the endpoints not being counted.
    \item Define a Laurent polynomial in variables $A,u$, $T_K(A,u)\in \Z[A^{\pm1},u^{\pm1}]$ as follows
\end{itemize}

$$
\displaystyle{T_K(A,u)=(-A^3)^{-w(K)}u^{-K \cdot \alpha}\sum_{s:state} A^{\sigma_s} u^{k_s \cdot \alpha} (-A^2-A^{-2})^{|s|-1}}
$$

As it turns out, $T_K(A,u)$ is an isotopy invariant for knotoids. 

\begin{rem}\rm
Clearly, the Kauffman bracket polynomial can be extracted from the Turaev extended bracket polynomial by specializing $u=1$, see also \cite{turaev2012knotoids}.
\end{rem}

\subsection{Finite type invariants and the Jones polynomial}

The search for finite type invariants coming out of known polynomial invariants has caught in the past the attention of many mathematicians starting with Birman \& Lin in \cite{Birman1993}. Some interesting calculations that give rise to powerful finite type invariants were given in \cite{HKAUFFMAN1987395,KAUFFMAN2004DIAGR}. 
Recall the definition of the normalized bracket for a knot $K$ to be 
$$ \displaystyle{f_K (A) = (-A^{-3})^{-w(K)} \left\langle K\right\rangle (A)} 
$$

where the writhe $w(K)=n_+ - n_-$ is this sum of crossing sums and $n_\pm$ is the number of positive/negative crossings.

Making the substitution $t^ {1/4} = A^{-1}$ we obtain the classical Jones polynomial.

The skein relation for the Jones polynomial together with an initial condition for the unknot is
$$t^{-1}J(\PC) - tJ(\NC) = (t^{1/2}-t^{-1/2})J(\TU)\ ;\qquad\qquad 
  J(\unkt) = 1\ . 
$$

Making a substitution $t:=e^x$ and
then taking the Taylor expansion into a formal power series in $x$, we can represent the Jones polynomial of a knot $K$ as a power series
$$
J(K)=\sum_{n=0}^\infty j_n(K) x^n\ .
$$
We claim that the coefficient $j_n(K)$ is a Vassiliev invariant of order 
$\le n$. Indeed, substituting $t=e^x$ into the skein relation gives
$$
(1-x+\dots)\cdot J\left(\PC\right) -
  (1+x+\dots)\cdot J\left(\NC\right) =
  (x+\frac{x^2}{4} +\dots)\cdot J\left(\TU\right)
$$
From which we get
$$
J\left(\SC\right) = J\left(\PC\right) - J\left(\NC\right) =
   x\Bigl(\mbox{some mess}\Bigr)
$$

The only summands that are not divisible by $x$ are on the left hand side. 
This means that the value of the Jones polynomial on a knot with a single double point is divisible by $x$. Therefore, the Jones polynomial of a singular knot with $k>n$ double points is divisible by $x^k$, and thus its $n$th coefficient vanishes on such a knot.

The exact same calculations apply in the case that $K$ is a knotoid diagram by taking as initial condition that  $J\left(\unkd\right)=1$. The interesting point here is that the type-1 invariant coming from the Jones polynomial seems to be trivial even for knotoids, since the bracket cannot see the winding around the leg. This question remains still open for us. In Subsection~\ref{sub:Tur}  we generalize this idea.










\subsection{Higher type invariants for knotoids coming from the Turaev extended bracket}\label{sub:Tur}

We now proceed with making in the Turaev extended bracket polynomial $T_K(A,u)$ the substitution $A=e^x$.

\begin{prop}
  Substituting  $A=e^x$ in the Turaev extended bracket polynomial $T_K(A,u)$ of a knotoid $K$ yields the   power series expansion: 
$$
\displaystyle{T(K)=\sum_{k=0}^{\infty}\sum_{l=-n_1}^{m_1}t_{k,l}u^lx^k}
$$
Then $\displaystyle{\sum_{l=-n_1}^{m_1}t_{k,l}u^l}$ and $t_{k,l}$ for all $l \in \{-n_1,-n_1+1,\dots,m_1\}$ are type-$k$ invariants of knotoids.
\end{prop}

\begin{proof}
We have the skein relation 
$$
\displaystyle{-A^4T\left(\PC\right)+A^{-4}T\left(\NC\right)=(A^2-A^{-2})T\left(\TU\right)=}
$$
\vspace{10pt}
$$
-\left(1+4x+\frac{16x^2}{2}+\dots\right)T\left(\PC\right)+\left(1-4x+\frac{16x^2}{2}-\dots\right)T\left(\NC\right)=
$$
\vspace{10pt}
$$
=\left((1+2x+\dots)-(1-2x+\dots)\right)T\left(\TU\right)
$$

Hence, by the Vassiliev skein relation
$$
T\left(\SC\right)=T\left(\PC\right)-T\left(\NC\right)=4x(\mbox{some mess}).
$$
So, $T\left(\SC\right)$ is divisible by $x$. Hence, if a knotoid has $n>k$ singularities, then the coefficient of $x^k$ vanishes.

So we have the following results:
\begin{enumerate}
    \item 
    $\displaystyle{\sum_{l=-n_1}^{m_1}t_{k,l}u^l}$ is a type-$k$  invariant.
    \item
    The fact that the coefficient of $x^k$ vanishes, means that for  knotoids with $n>k$ singularities we have:  $\displaystyle{\sum_{l=-n_1}^{m_1}t_{k,l}u^l=0} \quad \forall u$, which implies that $t_{k,l}=0 \quad \forall l \in \{-n_1,-n_1+1,\dots,m_1\}$.
    So, every $t_{k,l}$ is a type-$k$ invariant.    
\end{enumerate}
\end{proof}

\subsection{Some specific computations}

We shall now calculate some specific finite type invariants arising from

$$\displaystyle{f_K(A)=(-A^3)^{-w(K)} \left\langle K\right\rangle}= c_{-k}A^{-k}+\dots c_{m}A^m=\sum_{l=-k}^{m}c_{l}A^l$$
Expanding we obtain:
$$\displaystyle{f_K(e^x)=\sum_{l=-k}^{m}c_{l}e^{lx}=\sum_{l=-k}^{m}c_{l}\sum_{n=0}^{\infty}\frac{(xl)^n}{n!}=\sum_{n=0}^{\infty}\left(\frac{1}{n!}\sum_{l=-k}^{m}c_{l}l^n\right)x^n}$$

So, $$\displaystyle{v_n=\frac{1}{n!}\sum_{l=-k}^{m}c_{l}l^n}$$

Now we can calculate the finite type invariants arising from the bracket for the knotoid diagram of Fig.~\ref{fig:knotoid_calc}.

\begin{figure}[H]
    \centering
    \includegraphics[height=3cm]{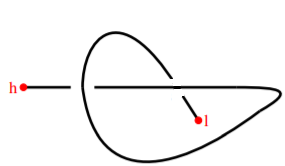}
    \caption{The knotoid diagram $K$}
    \label{fig:knotoid_calc}
\end{figure}

Following the definition of the normalized bracket $f_K$ we get
$$\displaystyle{\left\langle K\right\rangle=A^2(-A^{2}-A^{-2})+1+1+A^{-2}=-A^4+1+A^{-2}}$$

Hence: $$\displaystyle{f_K(A)=(-A^3)^{-w(K)} \left\langle K\right\rangle=A^6(-A^4+1+A^{-2})=-A^{10}+A^6+A^4}$$

From the previous conversation we have: 
$$\displaystyle{f_K(x)=\sum_{n=0}^{\infty}v_nx^n=\sum_{n=0}^{\infty}\frac{-10^n+6^n+4^n}{n!}x^n}$$
$$\displaystyle{v_1(K)=-10^1+6^1+4^1=0}$$
$$\displaystyle{v_2(K)=\frac{-10^2+6^2+4^2}{2!}=-24}$$
$$\displaystyle{v_3(K)=\frac{-10^3+6^3+4^3}{3!}=-120}$$

Now if $K^*$ is the mirror image of $K$ then: $$\displaystyle{f_{K^*}(A)=-A^{-10}+A^{-6}+A^{-4}}$$

and following in the same way

$$\displaystyle{v_1(K^*)=10^1-6^1-4^1=0}$$
$$\displaystyle{v_2(K^*)=\frac{-(-1)^210^{2}+(-1)^26^2+(-1)^24^2}{2!}=-24}$$
$$\displaystyle{v_3(K^*)=\frac{-(-1)^310^3+(-1)^36^3+(-1)^34^3}{3!}=120}$$
Hence we obtained:
\begin{cor}
 $v_3$ is the first finite type invariant arising  from the bracket that distinguishes the knotoid $K$ from its mirror image.
\end{cor}

The same calculations yield that $$\displaystyle{T_K(A,u)=A^4+(A^6-A^{10})u^2}$$
$$\displaystyle{T_K(x,u)=\sum_{n=0}^{\infty}v^{\prime}_nx^n=\sum_{n\in \mathbb{N}, k=0,2}t_{n,k}x^nu^k=\sum_{n=0}^{\infty}\left(\frac{4^n}{n!}+\frac{6^n-10^n}{n!}u^2\right)x^n}$$

Then $$\displaystyle{v^{\prime}_1(K)=4-4u^2}$$
$$\displaystyle{v^{\prime}_2(K)=\left(\frac{4^2}{2!}+\frac{6^2-10^2}{2!}u^2\right)=8-32u^2}$$

But also $t_{n,k}$ are finite type invariants, as we saw, and for example $t_{1,0}(K)=4,\quad t_{1,2}(K)=-4$ and they are non-trivial since the Turaev extended bracket can see the winding around the leg, expressed as intersection number.

\section*{Acknowledgement}

We would like to express our thanks to both Reviewers for their  important comments, which led to the improvement and clarity of the manuscript.

\nocite{*}
\bibliographystyle{IEEEannot}  

\bibliography{template}

\end{document}